\documentclass[a4paper,12pt]{article}

\usepackage[margin=0.7in]{geometry}
\usepackage{float}
\usepackage{color}
\usepackage{graphicx}
\usepackage{multirow}
\usepackage{hyperref}
\usepackage{amssymb}
\usepackage{amsmath}
\usepackage{caption}
\usepackage{subcaption}
\usepackage{textcomp}
\usepackage{harvard}
\usepackage{authblk}

\usepackage{eurosym}
\usepackage[authormarkup=none]{changes}

\newtheorem{proposition}{Proposition}
\newenvironment{proof}[1][Proof]{\textbf{#1.} }{\ \rule{0.5em}{0.5em}}

\renewcommand{\newtheorem}{\arabic{prediction}:}

\usepackage{xcolor}

\begin{document}

\title{\textbf{Feed-in Tariff Contract Schemes and Regulatory Uncertainty}}
%\date{\today}
\date{}

\author[a]{Luciana Barbosa}
\author[b]{Cl\'audia Nunes\footnote{Corresponding author:Cl\'audia Nunes, cnunes@math.tecnico.ulisboa.pt }}
\author[c]{Artur Rodrigues}
\author[d]{Alberto Sardinha}

\affil[ ]{\footnotesize}
\affil[a]{Universidade Aut\'{o}noma de Lisboa}
\affil[b]{CEMAT and Instituto Superior T\'{e}cnico, Universidade de Lisboa}
\affil[c]{NIPE and School of Economics and Management, University of Minho}
\affil[d]{INESC-ID and Instituto Superior T\'{e}cnico, Universidade de Lisboa}
%\affil[d]{INESC-ID and Instituto Superior T\'ecnico, Universidade de Lisboa}

\maketitle

\citationmode{full}

\begin{abstract}
\noindent This paper presents a novel analysis of two feed-in tariffs (FIT) under market and regulatory uncertainty, namely a sliding premium with cap and floor and a minimum price guarantee. Regulatory uncertainty is modeled with a Poisson process, whereby a jump event may reduce the tariff before the signature of the contract. Using a semi-analytical real options framework, we derive the project value, the optimal investment threshold, and the value of the investment opportunity for these schemes. Taking into consideration the optimal investment threshold, we also compare the two aforementioned FITs with the fixed-price FIT and the fixed-premium FIT, which are policy schemes that have been extensively studied in the literature. Our results show that increasing the likelihood of a jump event lowers the investment threshold for all the schemes; moreover, the investment threshold also decreases when the tariff reduction increases.  We also compare the four schemes in terms of the corresponding optimal investment thresholds. For example, we find that the investment threshold of the sliding premium  is lower than the minimum price guarantee. This result suggests that the first regime is a better policy than the latter because it accelerates the investment while avoiding excessive earnings to producers.
\end{abstract}
\noindent
 
\section{Introduction}

The quest for an optimal policy for incentivizing renewable energy production has led policymakers to create several support schemes over the past decades. Many jurisdictions use different instruments to support renewable energy production, such as feed-in tariffs, quota obligations with green-certificate trading, tenders, and tax exemptions.  A key aim of these support schemes is to accelerate the investment decision of renewable energy projects while making energy production more sustainable \cite{Couture10}.

According to the Renewables 2019 Global Status Report \cite{ren21_2019}, 111 jurisdictions were using feed-in tariff policies by the end of 2018, while 48 jurisdictions held auctions (i.e., tenders) in 2018, and 33 jurisdictions were using quota obligations with green-certificate trading in 2018. We can thus conclude that feed-in policies continue to play an important role in the policy-making landscape of renewable energies. 

Feed-in tariffs (FITs) are price-based schemes with long-term contracts, whereby an investor receives a fixed or variable remuneration for a long period (e.g., 15 - 20 years). According to \citeasnoun{Couture10}, the feed-in policies can thus be classified into two groups. First, the market-independent FITs are long-term contracts that pay remunerations that are not tied to the energy market price. These feed-in policies are suited to support small scale projects or less-mature technologies because they protect investors from price risk and thus reduce the cost of capital. The second group is the market-dependent FITs or feed-in premiums, whereby investors receive a fixed or variable premium over the energy market price. This category includes support schemes such as a fixed premium, a minimum price guarantee and a sliding premium with cap and floor\footnote{The feed-in premium payments can be designed to be either a fixed predetermined amount over the market price or sliding where the premium varies as a function of the market price.}.

The \citeasnoun{ec13} sent out a recommendation for policymakers to shift from the popular market-independent FITs to the feed-in premiums, because they believe that energy market signals should be the key drivers of investment and production decisions. Feed-in premiums can thus create incentives for producers to optimize investment decisions, plant design, and production according to market signals.

This work uses the real options framework to present a novel analysis of two feed-in premiums, namely a FIT with a minimum price guarantee (i.e., price-floor regime) and a FIT with a sliding premium with cap and floor (i.e., collar regime). We also compare these two feed-in premiums with two popular feed-in tariffs, namely a fixed price and fixed premium, that serve as baselines for our analysis. The aim is to compare the key aspects of these support schemes that trigger a renewable energy investment  not only from an investor's perspective but also from a policymaker's viewpoint. Netherlands \cite{Agnolucci07}, Ireland \cite{Doherty11}, and Switzerland \cite{Couture10b} are examples of jurisdictions that have used variations of the price-floor regime within a FIT contract. In addition, Spain \cite{Barcelona12} is an example of a jurisdiction that used the sliding premium with cap and floor. 

Since its inception in 1978, policy revisions of FITs have occurred due to changing market conditions and the need to meet specific goals of each jurisdiction. \citeasnoun{Ritzenhofen16} state that policymakers adjust feed-in policies in order to accelerate renewable energy investments and to find cost-effective policies that are subject to budget constraints. Hence, we also analyze the impact of policy revisions on the optimal investment timing of feed-in premiums.

This work aims to shed some light on some challenging policy-making decisions. In particular, we focus on the following research questions:

\begin{itemize}
    \item How does a FIT with a minimum price guarantee and a FIT with a sliding premium with cap and floor affect the timing of an investment decision, especially in comparison with other feed-in policies that have been extensively studied in the literature?
    \item How does regulatory uncertainty impact the timing of an investment decision when policymakers may reduce the tariff of these two feed-in premiums?
\end{itemize}

Within the scientific literature, \citeasnoun{Kim12} also compare four different feed-in policies, namely the fixed price, the fixed premium, the minimum price guarantee, and the sliding premium with cap and floor. However, the work focuses on a different optimization problem, where the authors optimize the tariffs in order to maximize the number of sign-ups. In contrast, our work tackles a different optimization problem, where we consider a fixed tariff and derive the optimal investment thresholds. In the same line of research, the paper from \citeasnoun{Chronopoulos16} focuses on the same optimization problem as our work, where policy uncertainty is also taken into consideration. However, the work only analyzes a fixed-premium FIT and not the minimum price guarantee and the sliding premium with cap and floor. We thus build on these previous research works to present a more complete analysis of FITs and how the optimal investment timing can change under regulatory uncertainty.

Comparing our results with the work from \citeasnoun{Chronopoulos16}, we also find that investment accelerates when the fixed-premium FIT has a higher probability of retraction. In addition, we show that this results holds for the other feed-in policies, because investors anticipate the decision in order to obtain a higher tariff. Our results also show that the risk of a tariff reduction has a higher impact on the optimal investment timing of the fixed-price and the fixed-premium regimes than the price-floor (i.e., minimum price guarantee) and the collar (i.e., sliding premium with cap and floor) regimes. We also present more results aimed at policy-making decisions, when comparing the optimal investment thresholds of the FIT schemes. For instance, the investment threshold of the sliding premium is lower than the threshold of the  minimum price guarantee. This result suggests that the sliding premium is a better policy than the minimum price guarantee because the former policy scheme accelerates the investment while avoiding excessive earnings\footnote{Many works that have analyzed a FIT with a sliding premium with cap and floor state that this policy scheme avoids excessive earnings to producers, such as the work from \citeasnoun{Cossent11}.} to producers.

The remainder of this paper is organized as follows. Section \ref{sec:relWork} presents the literature review and our main contributions to literature. In Section \ref{sec:opt}, we formalize the optimization problem that this paper addresses, and Section \ref{sec:sliding_premium} presents the model for the sliding premium with cap and floor. In Section \ref{sec:otherFITs}, we then present the models for the other FIT schemes, namely the minimum price guarantee, the fixed price, and the fixed premium. Section \ref{sec:Comparative statics analysis Trigger} presents the analytical and numerical study of the investment thresholds for all types of FIT design. Finally, we present the concluding remarks in Section \ref{sec:conclusions}.

\section{Literature review}
\label{sec:relWork}

\citeasnoun{Myers77} was the first scholar to coin the term ``real options'' with an aim to create a methodology to better value investment projects in the presence of managerial flexibilities. \citeasnoun{Tourinho79} presented the first real options framework which analyzed the option to extract oil when future prices are uncertain. A few years later, \citeasnoun{Brennan85} used the real options framework to analyze mineral extraction investments. In addition, \citeasnoun{McDonald86} derived the optimal investment rule which takes into account the value of waiting. The authors show that the standard net present value (NPV) analysis is grossly wrong.

Many managers in the corporate world rely on the NPV rule to analyze investment projects. A project with a positive NPV is considered to have a profitable outcome. In contrast, a negative NPV is considered to result in a net loss. However, the standard NPV valuation does not take into account uncertainties and managerial flexibilities that are present in dynamic environments; thus, it underestimates the value of the project. For example, the NPV method does not add the value of several ``options'' (i.e., managerial flexibilities) that managers use in practice, such as gathering information from the market and waiting to invest until market conditions are favorable. Consequently, a project with a positive NPV may be optimally deferred, or  a project with a negative NPV may lead to a profitable outcome due to the added value of these options. For instance, the case study from \citeasnoun{Rocha07} shows how the added value of these options can impact an investment project. The real options approach can thus provide a good framework to value investment in the presence of uncertainties and managerial flexibilities, such as renewable energy projects.

According to \citeasnoun{Cesena13}, the real options approach can be very useful for the decision-making process of renewable energy investments. Renewable energy projects are irreversible and present several uncertainties, such as market price uncertainty and regulatory uncertainty. In addition, managers have several options in the decision-making process, such as the waiting option. %The authors state that many research opportunities are still available for applying the real options methodology within the renewable energy field.

In fact, many scholars have used the real options approach to analyze renewable energy technologies in the scientific literature. For instance, \citeasnoun{Shao17} proposed a framework for policymakers that finds the optimal subsidies or optimal price discount rates to incentivize the adoption of electric vehicles. \citeasnoun{Bonneuil16} analyze the optimal moment for adopting renewable energies and the optimal rate of adoption. \citeasnoun{Kowalski09} use a participatory multi-criteria analysis to assess different renewable energy scenarios of technology deployment. However, these works do not analyze the feed-in tariffs as an incentive to accelerate the investment decision of renewable energy projects.

\citeasnoun{Doherty11} use a real options approach to analyze the Irish REFIT program, which is a variation of a FIT with a minimum price guarantee. With a Generalized Extreme Value distribution, the authors forecast the mean expected market value of wind and the value of the put option, given a strike price, in order to analyze the efficiency of the Irish FIT scheme in comparison with the fixed-price FIT. The results show that the Irish FIT scheme creates a similar or lower incentive than the fixed-price FIT for wind projects.  \citeasnoun{Kim12} present an analysis of four different FIT schemes, namely fixed price, fixed premium, minimum price guarantee, sliding premium with cap and floor. In particular, the work uses a binomial lattice and simulation (i.e., real options numerical technique) to compare the four different FITs, whereby the work focuses on how different scenarios (e.g., markets with different price volatilities) can impact the number of sign-ups of renewable energy projects. The results suggest that the number of sign-ups increases in more volatile markets, to the detriment of a higher cost for ratepayers. While these two research works analyze FIT schemes from an efficiency viewpoint, they are based on a different optimization problem. In particular, these works do not analyze the optimal investment thresholds and the influence of regulatory uncertainty on the timing of the investment decision.

\citeasnoun{Boomsma12} examine renewable energy investments under different policy schemes, namely a fixed-price FIT, a fixed-premium FIT, and a renewable energy certificate trading scheme. With a semi-analytical framework, the authors derive the optimal investment threshold and the optimal capacity to install. The work also assumes multiple sources of uncertainty (e.g.,  steel spot price and subsidy payment) and uncertainty with respect to any change of support scheme. The results show that fixed-price FITs create incentives for speeding up investment while the certificate trading generates larger projects. Following a similar approach, \citeasnoun{Boomsma15} employ  a semi-analytical real options framework to analyze policy uncertainty for the same  policy schemes in \protect \citeasnoun{Boomsma12}.  The policy uncertainty considers that a scheme may be terminated at any point in time. The results show that policy uncertainty increases the investment threshold when the termination decision is retroactively applied. In contrast, investment is accelerated when the decision is not retroactively applied. 
Our work is significantly different from \citeasnoun{Boomsma12} and \protect \citeasnoun{Boomsma15} because we derive two new support schemes, namely the minimum price guarantee and sliding premium. In addition, our  numerical study has novel results for policymakers regarding the comparison of the optimal investment thresholds of four FIT schemes. 

\citeasnoun{Ritzenhofen16} use a real options approach to analyze the effect of regulatory uncertainty on investment behavior, where a policymaker may decide to switch from a fixed-price FIT to a free-market regime (or the reverse case). The research work shows that regulatory uncertainty delays the investment when the FIT tariff is close to the market price. If the FIT tariff is significantly higher than the market price, then the regulatory uncertainty accelerates investment behavior. With a semi-analytical real options framework, \citeasnoun{Chronopoulos16} also analyze the same optimization problem as our work. The model includes a fixed-premium FIT where investment may occur within a single stage or through multiple stages. The work uses a Poisson process to model the policy uncertainty, whereby the FIT may be introduced or retracted multiple times. The results show that investment accelerates when the fixed-premium FIT has a higher probability of retraction, which is the same result in our work. However, we show that this result also holds to the other FIT schemes. In addition, the authors find that a permanent subsidy retraction reduces the amount of capacity installed. In summary, these two research works analyze a renewable energy investment under regulatory uncertainty; however, our work has several differences. First, we analyze two new FITs, namely a minimum price guarantee and a sliding premium, which are not included in the analysis of these two papers. Second, we consider that a policymaker may reduce the tariff instead of a regime-switching scenario of subsidies. Lastly, we present several policy-related results regarding the optimal investment thresholds of the aforementioned feed-in premiums that are not discussed in these two papers.

In summary, to the best of our knowledge, our paper presents a novel analysis of the optimal investment thresholds of finite-lived FITs with a minimum price guarantee and a sliding premium under regulatory uncertainty, whereby a policymaker may reduce the tariff before the signature of the contract. With  a semi-analytical real options framework, we derive the value of the project, the option value, and the optimal investment threshold of each FIT scheme. In addition, we also compare these two FITs with a fixed-price FIT and fixed-premium FIT in order to draw conclusions for both investors and policymakers.

\section{Optimization problem}
\label{sec:opt}

Throughout the paper, we assume that the energy market price changes stochastically over time, under the risk-neutral measure. We denote it by $P=\{P_t, t \geqslant 0\}$, where $P$ is the solution of the following stochastic differential equation:

\begin{equation}
\label{eq:gbm}
    dP_t=\mu P_t dt + \sigma P_tdW_t
\end{equation}
with $W=\{W_t, t \geqslant 0\}$ being a Brownian motion, $\mu$  the risk-neutral drift, and $\sigma$  the volatility. 

Note that we assume that energy market prices follow a geometric Brownian motion and this can be considered a simplifying assumption. However, if we consider other dynamics - like a mean-reverting process - the change in the qualitative results would be small, according to \protect \citeasnoun{schwartz1998valuing}. Moreover, as suggested by \protect \citeasnoun{Boomsma12}, the drift $\mu$ can be estimated from forward/futures prices.

We assume that before investment the firm has a zero payoff and that the investment cost is $I$. We let $r$  denote the interest rate under the risk-neutral measure. In order to avoid trivial or ill-posed problems, we assume that $\mu<r$. The optimization problem that we want to address is the following:
\begin{equation}
\label{eq:VF}
F_S(P)=\sup_{\tau} E \left[  \int_\tau^{\infty} \pi_S(P_t)e^{-rt}dt - Ie^{-r\tau}|P_0=P\right]
\end{equation}
where $\pi_S$ denotes the profit function upon investment, that we assume to be of the following type:
\begin{equation}
\pi_S(P_t)=\begin{cases} \Pi_S(P_t) & \tau \leqslant t \leqslant \tau+T\\ P_tQ & t>\tau+T
\end{cases}
\end{equation}
with $\tau$ being the time at which investment takes place, $\Pi_S$ is the profit of the firm as long as the subsidy is received, and  $T$ is the duration of the contract. 

Throughout this paper, we consider four types of subsidy schemes, which lead to different investment strategies of the firm. Therefore, we use the subscript $S$ to denote the particular subsidy that we are considering. The subsidies impact  the profit function $\Pi_S$ of the firm as we will see.  Furthermore, we assume that after the subsidy has been removed, the profit of the firm depends solely on the market price and the quantity $Q$ (total annual production), which we assume to be fixed over time. We note that  similar assumptions have been used in \citeasnoun{Boomsma12} and \protect \citeasnoun{Boomsma15}, for instance.\footnote{In their papers, the authors assume that the production is fixed, depending only on the capacity installed by the firm. Although for short intervals of time, production of renewable energy sources vary with weather conditions, for longer intervals the production is more stable.}

%\mynotes{(the investment starts at $\tau$, right? And we have used $T$ as the duration of the contract; thus, the subsidy ends at $\tau+T$)} 

If the firm invests at time $\tau$ when the current energy market price is $P$, then it will receive the following amount:
\begin{equation}
\label{eq:VS}
V_S(P)=E\left[\int_\tau^{\tau+T} \Pi_S(P_t)e^{-rt}dt+\int_{\tau+T}^{\infty} P_tQe^{-rt}dt|P_0=P\right].
\end{equation}

%\mynotes{(Does the integral above start at 0 or $\tau$? I have also changed $T$ in the integral to $\tau+T$)}

Using (\ref{eq:VS}) in (\ref{eq:VF}), we conclude that the optimization problem that we propose to solve is the following:
\begin{equation}
\label{eq:VFVS}
F_S(P)=\sup_{\tau} E \left[ \int_\tau^{\tau+T}\Pi_s(P_t)e^{-rt}dt+\int_{\tau+T}^
{\infty} P_tQe^{-rt}dt - Ie^{-r\tau} |P_0=P\right].
\end{equation}

The problem described in (\ref{eq:VFVS}) is a standard investment problem, which can be solved using the dynamic programming principle. We refer, for instance, to \protect\citeasnoun{dixit1994investment}, for the main results about investment problems in a context of real options. Applying such results leads to the following (general) solution of (\ref{eq:VFVS}): 

\begin{equation}
\label{eq:FS}
F_S(P)=\begin{cases} A_1 P^{\beta_1} & 0<P<P^{\ast}\\
V_S(P)-I & P\geqslant P^{\ast}
\end{cases}
\end{equation}
with 
\begin{equation}
\label{eq:quadratic_beta1}
\beta_1 = \dfrac{1}{2} - \dfrac{\mu}{\sigma^2} + \left( \left(-\dfrac{1}{2} + \dfrac{\mu}{\sigma^2} \right)^2 + \dfrac{2 r}{\sigma^2} \right)^\frac{1}{2}>1
\end{equation}
and $A_1$ and $P^{\ast}$ being such that $A_1 {P^{\ast}}^{\beta_1} =V_S(P^{\ast})-I$ and $\beta_1 A_1 {P^{\ast}}^{\beta_1-1}=V_S'(P^{\ast})$ (i.e., the so-called value matching and smooth pasting conditions). 

In the remainder of the paper, we let $F$ denote the feed-in tariff and assume it to be fixed.  The four different types of feed-in tariff contracts that we consider in the paper are the following:

\begin{itemize}
\item Sliding premium with cap and floor, with $\Pi_C(P):=\min\left(\max(P,F),C \right)Q$; in this particular case, the tariff $F$ corresponds to the price floor, $C$ is the price cap, and $C\geqslant F$.
\item Minimum price guarantee, with $\Pi_M(P):=\max(P,F)Q$.
\item Fixed-price scheme, for which $\Pi_F(P):=FQ$.
\item Fixed-premium scheme, for which $\Pi_P(P):=(P+F)Q$.
\end{itemize}

Moreover, in this paper, we extend the framework to include policy uncertainty, whereby a downward jump in the tariff $F$ may occur according to a Poisson process. Note that the Poisson process was also used in \protect \citeasnoun{Boomsma15} and \protect \citeasnoun{Chronopoulos16} to model policy uncertainty\footnote{Alternative approaches, such as regime switching, have been used by \protect \citeasnoun{Boomsma12} and  \protect \citeasnoun{Ritzenhofen16}.}.

Let $Y$ denote the time at which this change in the tariff occurs; in view of the assumption of a Poisson process, $Y$ follows an exponential random variable with parameter $\lambda$. We assume, furthermore, that the price process in (\ref{eq:gbm}) and the time at which this tariff change occurs are uncorrelated. Moreover, the tariff decreases by multiplying the tariff $F$ with $\omega$ (i.e., the new tariff is equal to $\omega F$), where $0 \leqslant \omega \leqslant 1$ . 

In order to take into account the reduction of the tariff, we change slightly the notation, where $\Pi_S(P,F)$ and $\pi_S(P,F)$, for $S \in \{M,F,P\}$\footnote{$M, F, P$ corresponds to minimum price guarantee, fixed price, and fixed premium respectively.}, denote the profit flow under regime $S$ when the tariff received is equal to $F$. In addition, $\Pi_S(P,\omega F)$ denotes the corresponding change in the profit flow when there is a reduction of the subsidy, changing from $F$ to $\omega F$. In case of retraction of a subsidy (i.e., the subsidy is totally abolished), $\Pi_S(P,0)$ denotes the change in the profit flow.

For the sliding premium case, we also use a similar notation as above. Since this subsidy has a cap ($C$) and a floor ($F$), we denote the profit flow by $\Pi_C(P, F,C)$. Thus, $\Pi_C(P, \omega F, \omega_C C)$ is the profit flow when the price floor is reduced to $\omega F$ and the price cap to $\omega_C C$.

We assume that once the firm invests, there will not be any change in the FIT contract (i.e., the tariff $F$ does not change after the firm invests). Therefore, the optimization problem may be written as follows:
\begin{equation}
F_{SR}(P)=\sup_{\tau} E \left[ \int_{\tau}^{\tau+T}e^{-r t}\left(\pi_S(P_t,\omega F)1_{\{Y<\tau\}}+\pi_S(P_t,F)1_{\{Y>\tau\}}\right)dt+\frac{PQ}{r-\mu}e^{-(r-\mu)(\tau+T)}\right]
\end{equation}
with $1_{\{Y<\tau\}}=1$ if $Y<\tau$ and $1_{\{Y>\tau\}}=1$ if $Y>\tau$. We use the notation  $F_{SR}$ in order to distinguish the case of {\bf R}egulatory uncertainty. Therefore, once the firm decides to undertake the investment, we assume that there is no uncertainty about the tariff contract.

With these remarks, we proceed with the analysis of the problem. We need to consider two cases:
\begin{itemize}
\item[i)] The tariff has already decreased, thus the firm knows that the tariff is $\omega F$ during a period $T$. Therefore, we denote $F^{(\omega)}_{S}(  P )$ as a value function equal to $F_S$, as defined in (\ref{eq:FS}), where we substitute $\omega F$ for $F$.

\item[ii)] At the current time, the value of the tariff has not been subject to a decrease. In this case, there is uncertainty about the evolution of the prices and the time at which the reduction of the subsidy occurs (relevant only before the investment time).  Thus, the corresponding  infinitesimal generator is now  of the following form:
\begin{equation}
\label{IG}
{\cal L}_{SR}F_{FR}(P)=\mu \:P \dfrac{\partial F_{SR}( P )}{\partial P}+0.5 \sigma^2 \: P^2 \: \dfrac{\partial^2 F_{SR}( P )}{\partial P^2}+\lambda[F^{(\omega)}_{S}(P) - F_{SR}(P)]
\end{equation}
where $1/\lambda$ denotes the expected time until the change in the policy occurs; see \citeasnoun{oksendal2005applied} for further details.

In \eqref{IG}, the term $[F^{(\omega)}_{S}(  P ) - F_{SR}( P )]$ accounts for the change in the value in case the tariff decreases, which, in an infinitesimal interval, occurs with probability $\lambda dt$. Note that this term is precisely the major difference with the case where there is not regulatory uncertainty. Moreover, for $P$ in the continuation region (i.e, for values of $P$ where investment is not yet optimal), the value function is solution of the following equation:

\begin{equation}
\mu \:P \dfrac{\partial F_{SR}( P )}{\partial P}+0.5 \sigma^2 \: P^2 \: \dfrac{\partial^2 F_{SR}( P )}{\partial P^2}-rF_{SR}( P ) +\lambda[F^{(\omega)}_{S}(P) - F_{SR}( P )] = 0
\end{equation}

Recall that $F^{(\omega)}_{S}(P)$ is (\ref{eq:FS}) substituting $\omega F$ for $F$. As $\omega$ is smaller than one, the investment trigger with $\omega F$ is greater than the investment trigger with $F$, because if one does not invest when the tariff is $F$, he will not invest when the tariff is $\omega F$ either. Therefore, the second branch in (\ref{eq:FS}) will never occur. And consequently, when the tariff is $\omega F$, $P$ is also in the continuation region. $F^{(\omega)}_{S}( P )$ is then equal to $A_1^{(\omega)} P^{\beta_1}$, and the value function before the investment must be the solution of the following equation:

\begin{equation}
\label{eq:ODE_fixed_ru2}
\mu \:P \dfrac{\partial F_{SR}( P )}{\partial P}+0.5 \sigma^2 \: P^2 \: \dfrac{\partial^2 F_{SR}( P )}{\partial P^2}-rF_{SR}( P ) +\lambda[A_1^{(\omega)} P^{\beta_1} - F_{SR}( P )] = 0
\end{equation}

In general, the solution of (\ref{eq:ODE_fixed_ru2}) is of the following form \cite{nunes2017analytical}:
\begin{equation}
\label{eq:Ffixed_RU}
F_{SR}( P ) = B_1 P^{\eta_1} + A_1^{(\omega)}  P^{\beta_1} 
\end{equation}
where $\eta_1$ is the positive root of the quadratic equation:
\begin{equation}
\label{eq:etas}
0.5\sigma^2\eta(\eta-1)+\mu\eta-(r + \lambda) = 0.
\end{equation}
Hence,
\begin{equation}
\label{eq:quadratic_eta1}
\eta_1 = \dfrac{1}{2} - \dfrac{\mu}{\sigma^2} + \left( \left(-\dfrac{1}{2} + \dfrac{\mu}{\sigma^2} \right)^2 + \dfrac{2 (r + \lambda)}{\sigma^2} \right)^\frac{1}{2}.
\end{equation}

Therefore, in case of regulatory uncertainty, the value function takes the following form:
\begin{equation}
\label{eq:FSR}
F_{SR}(P)=\begin{cases}  B_1 P^{\eta_1} + A_1^{(\omega)}  {P_{SR}}^{\beta_1} 
 & 0<P<P_{SR}^{\ast}\\
V_S(P)-I & P\geqslant P_{SR}^{\ast}
\end{cases}
\end{equation}
where $P_{SR}^{\ast} $ denotes the investment threshold when the feed-in tariff used is $S$ and under regulatory uncertainty.
\end{itemize}

The constants $A_1$ and $A_1^{(\omega)} $ are found, in general, by fit conditions, so that $F_S$ and $F_{SR}$ are continuous functions, with first order 
derivative for all possible values of $P$ (see \citeasnoun{oksendal2005applied}).

Next, we derive the value function and the investment thresholds for the aforementioned FIT schemes, without and with regulatory uncertainty, using the results of this section. The minimum price guarantee is a special case of the sliding premium with cap and floor, when the cap $C\to \infty$. Moreover, the fixed price is also special case of the sliding premium, when the value of the cap is equal to the value of the floor (i.e., $C=F$). For this reason, as we will see in the following sections, we derive first the results for the sliding premium with cap and floor, and then use these results to derive the other two regimes. 

% Fim da adição de Cláudia

\section{Sliding premium with cap and floor FIT}
\label{sec:sliding_premium}

In this section, we derive the value function and the investment trigger for the sliding premium with cap and floor, also known as a collar regime. A collar regime has the following remuneration conditions: (i) the producer receives a price floor $F$ when the market price is below the price floor, (ii) the producer receives a price cap $C$, where $C \geqslant F$, when the market price is above the price cap, and (iii) the producer sells energy for the electricity market price $P$ when the market price lies between $F$ and $C$. 

Regarding the notation, we denote the payoff upon investment, the value function, and the investment threshold by $V_C, F_C$ and $P_C^{\ast}$, respectively.  As we also include the derivation of the same quantities under regulatory uncertainty, we use $V_C^{(\omega)}, F_C^{(\omega)}$, and $P_C^{\ast (\omega)}$ to denote the payoff upon investment, the value function, and the investment threshold when the price floor is $\omega F$ and the price cap is $\omega_C C$. Finally, $F_{CR}$ and $P_{CR}^{\ast}$ denote the value function and investment threshold under regulatory uncertainty.

The instantaneous profit function of the renewable energy project with a sliding premium  is given by:

\begin{equation}
\Pi_C(P) = \text{Min}(\text{Max}(P,F), C) Q 
\end{equation}

Next, we present the results concerning the investment thresholds for both cases, namely without and with regulatory uncertainty.

\subsection{No regulatory uncertainty}
We build on the work from \citeasnoun{adkins2019investment} in order to derive the sliding premium  and extend it in order to model a policy change before the project starts, with a Poisson process. In the above paper, the authors provide the value of  the project with a perpetual collar contract, given by:
\begin{equation}
\label{eq:ProjectValueC}
V_{CP}(P) =  
\begin{cases}
E_1 P^{\beta_1} + \dfrac{F Q }{r} & \textrm{ for }  P < F  \\ \\
G_1 P^{\beta_1} + G_2 P^{\beta_2} + \dfrac{P Q}{r - \mu} & \textrm{ for }  F \leqslant P < C  \\ \\
H_2 P^{\beta_2} + \dfrac{C Q }{r} & \textrm{ for }  P \geqslant C
\end{cases} 
\end{equation}
where:
\begin{align}
\label{eq:E1}
E_1 &= \dfrac{(F^{1-\beta_1}-C^{1-\beta_1})Q}{\beta_1-\beta_2} \left(\dfrac{\beta_2} {r}-\dfrac{\beta_2 - 1} {r - \mu}\right) <0 \\
\label{eq:L1}
G_1 &= -\dfrac{C^{1 - \beta_1} Q } {\beta_1 - \beta_2} \left(\dfrac{\beta_2} {r}-\dfrac{\beta_2 - 1} {r - \mu}\right) <0 \\
\label{eq:G2}
G_2 &= \dfrac{F^{1-\beta_2} Q } {\beta_1 - \beta_2}\left(\dfrac{\beta_1} {r} - \dfrac{\beta_1 - 1} {r - \mu}\right)>0 \\
\label{eq:M2}
H_2 &= \frac{(F ^{1 -\beta_2}-C^{1-\beta_2})Q}{\beta_1-\beta_2} \left(\dfrac{\beta_1} {r} - \dfrac{\beta_1 - 1} {r - \mu}\right) <0
\end{align}

We calculate $E_1$, $G_1$, $G_2$, and $H_2$ by equating the values and derivatives because $V_{MP}(P)$ must be continuously differentiable along $P$. In addition, $\beta_1$ is presented in \eqref{eq:quadratic_beta1} and $\beta_2$ is: 

\begin{equation}
\label{eq:quadratic_beta2}
\beta_2 = \dfrac{1}{2} - \dfrac{\mu}{\sigma^2} - \left( \left(-\dfrac{1}{2} + \dfrac{\mu}{\sigma^2} \right)^2 + \dfrac{2 r}{\sigma^2} \right)^\frac{1}{2}<0.
\end{equation}
Moving on to the finite case ($T<\infty$), we start by deriving the value received upon investment (i.e, the value of the project), as defined in (\ref{eq:VS}).

\begin{proposition}
\label{propVC}
The value of the project with a finite collar scheme is given by: 
\begin{equation}
\label{eq:VALC}
V_{C} (P) = V_{CP}(P) - S_C(P) + \dfrac{P Q}{r-\mu}e^{-(r-\mu)T}
\end{equation}
where 
\begin{multline}
\label{eq:SC}
S_C(P) =  E_1 P^{\beta_1} \Phi(-d_{\beta_1}(P,F)) + \dfrac{F Q }{r}e^{-rT} \Phi(-d_0(P,F)) \\
+ G_1 P^{\beta_1}(\Phi(d_{\beta_1}(P,F)) - \Phi(d_{\beta_1}(P,C))) 
+ G_2 P^{\beta_2}(\Phi(d_{\beta_2}(P,F)) - \Phi(d_{\beta_2}(P,C))) \\ + \dfrac{P Q}{r - \mu}e^{-(r - \mu)T}(\Phi(d_{1}(P,F)) - \Phi(d_{1}(P,C))) \\
+ H_2 P^{\beta_2}\Phi(d_{\beta_2}(P,C)) + \dfrac{C Q }{r}e^{-rT}\Phi(d_0(P,C))
\end{multline}
where, in the above equations, $\Phi(.)$ denotes the standard normal cumulative distribution, with
\begin{align}
\label{eq:d_beta}
d_{\beta}(P,X) &= \dfrac{\ln{\dfrac{P}{X}}+ \left(\mu + \sigma^2 \left(\beta -\dfrac{1}{2}\right)\right)T}{\sigma\sqrt{ T }},  \text{for } \beta \in \{0,1,\beta_1,\beta_2\}.
\end{align}
\end{proposition}
\begin{proof}
The value of the project with finite duration is equal to the value of the project with perpetual collar minus the value of the project with a delayed perpetual collar   (a collar that starts in the future moment $T$) plus the expected value of the total profit after the end of the subsidized contract. 

Equation (\ref{eq:SC}) corresponds to the second term. For more details see the work from \citeasnoun{adkins2019investment}.
\end{proof}

\bigskip

Therefore, we can present the value of the option and the investment threshold.

\begin{proposition}
\label{prop3}
The value of the investment option in a project with a sliding premium is given by:
\begin{equation}
\label{eq:optionvalueC}
F_{C} (P) =  
\begin{cases}
(V_{C}(P_C^{\ast}) - I)\left(\dfrac{P}{P_C^{\ast}}\right)^{\beta_1} & \textrm{ for } P < P_C^{\ast} \\ \\
V_{C}(P) - I & \textrm{ for } P \geqslant P_C^{\ast}\end{cases}
\end{equation}  
where the investment threshold, $P_C^{\ast}$, is the numerical solution of the following equation:
\begin{equation}
\label{eq:trigger_collar}
\beta_1 (V_C(P_C^{\ast}) - I) - \dfrac{\partial V_C( P_C^{\ast} )}{\partial P_C^{\ast}}P_C^{\ast} = 0
\end{equation}
\end{proposition}

This equation results in three equations that are presented in Appendix \ref{app:collar}. The solution of the investment threshold $P_{C}^\ast$, computed using (\ref{eq:trigger_collar}), can only be found using numerical methods. Therefore, one may ask about the existence of admissible and uniqueness of these solutions. We are not able to provide a formal proof, but we have tested a range of different values for the parameters, such as the base-case parameters in Section \ref{sec:Comparative statics analysis Trigger}, and we were always able to find numerical solutions with the desired accuracy. This is also the case for the other regimes, with and without regulatory uncertainty.  Although these equations are, in general, highly non-linear, again our tests suggest that any mathematics solver can find numerical solutions.

\subsection{Regulatory uncertainty}
Next, we follow similar steps to derive the investment threshold, but now for the scenario with regulatory uncertainty. According to (\ref{eq:ODE_fixed_ru2}), we need to solve the following differential equation:
\begin{equation}
\label{eq:ODE_MCollar}
\mu \:P \dfrac{\partial F_{CR}( P )}{\partial P}+0.5 \sigma^2 \: P^2 \: \dfrac{\partial^2 F_{CR}( P )}{\partial P^2}-rF_{CR}( P ) +\lambda[ J_1^{(\omega)} P^{\beta_1} - F_{CR}( P )] = 0.
\end{equation}
where $J_1^{(\omega)}$ refers to the first branch of (\ref{eq:optionvalueC}), which accounts for the value of the investment option in a project with a sliding premium with cap $\omega_c C$ and floor $\omega F$. Hence:
\begin{equation}
\label{eq:J1_omega}
J_1^{(\omega)} = (V_{C}^{(\omega)}(P_C^{\ast (\omega)}) - I)\left(\dfrac{1}{P_C^{\ast (\omega)}}\right)^{\beta_1}
\end{equation}
where $P_C^{\ast (\omega)}$ is the correspondent of $P_C^{\ast}$, given by (\ref{eq:trigger_collar}), when the price floor is $\omega F$ and the price cap is $\omega_C C$.  The general solution of \eqref{eq:ODE_MCollar} is $K_1 P^{\eta_1} + J_1^{(\omega)} P^{\beta_1}$.

Note that $\omega$ and $\omega_C$ are smaller or equal to 1, consequently $\omega F \leqslant F$ and  $\omega_C C \leqslant C$. We can thus conclude that $P_C^{\ast} < P_C^{\ast (\omega)}$. In other words, a smaller price floor and price cap leads to a higher investment threshold, because an investor that is still waiting to invest when the price floor is $F$ and the price cap is $C$, will certainly be waiting to invest when the price floor and the price cap are lower (i.e., equal to $\omega F$ and $\omega_C C$).

Finally, we are able to present the value function and the investment threshold for this case. 

\begin{proposition}
The value of the option with a collar FIT scheme and regulatory uncertainty is given by:
\begin{equation}
\label{eq:option_floor_reg}
F_{CR} (P) =  
\begin{cases}
(V_{C}(P_{CR}^{\ast}) - I - J_1^{(\omega)} {P_{CR}^\ast}^{\beta_1})\left(\dfrac{P}{P_{CR}^{\ast}}\right)^{\eta_{1}} + J_1^{(\omega)} P^{\beta_1}   & \textrm{ for } P < P_{CR}^{\ast}  \\
\ \\
V_C (P) - I & \textrm{ for } P \geqslant P_{CR}^{\ast}   
\end{cases}
\end{equation}
The investment threshold $P_{CR}^{\ast}$ is the solution of the following equation, that must be solved numerically:
\begin{equation}
\label{eq:trigger_collar_reg}
\eta_1(V_C(P_{CR}^{\ast}) - I - J_1^{(\omega)} {P_{CR}^\ast}^{\beta_1}) + \beta_1 J_1^{(\omega)} {P_{CR}^\ast}^{\beta_1} - \dfrac{\partial V_C( P_{CR}^{\ast} )}{\partial P_{CR}^{\ast}}P_{CR}^{\ast} = 0
\end{equation}
This equation results in three equations that are presented in Appendix \ref{app:collar}.
\end{proposition}

\section{Other FITs}
\label{sec:otherFITs}

Using the results derived in the previous section, we show next how we can derive the value function and investment threshold for the minimum price guarantee and the fixed-price FIT. Since the fixed-premium FIT is not a special case of the collar regime, we present the derivations independently at the end of this section.

In Appendix \ref{app:proofs_relations}, we present the proofs that the minimum price guarantee and the fixed-price FIT are a special case of the sliding premium, letting $C \to \infty$ for the minimum price guarantee case and $C=F$ for the fixed-price FIT case. 

\subsection{Minimum price guarantee FIT}

A FIT with a minimum price guarantee \cite{Barbosa2018} uses a price-floor regime. In this scheme, an investor either (i)  receives a fixed amount when the market price is below the price floor, or (ii) sells energy for the market price when the market price is above the price floor. As mentioned above, the price-floor regime is a special case of the collar regime when $C$ goes to  $\infty$.

The instantaneous profit function of the renewable energy project with a minimum price guarantee is:
\begin{equation}
\label{PiM}
\Pi_M(P) = \text{Max}(P,F) Q 
\end{equation}
where $F$ is a minimum price guarantee due to a FIT contract, and $Q$ is the quantity of energy produced.

Following the same reasoning of the collar regime, the value of the project with a finite minimum guarantee is given by:
\begin{equation}
\label{eq:VALM}
V_{M} (P) = V_{MP}(P) - S_M(P) + \dfrac{P Q}{r-\mu}e^{-(r-\mu)T}
\end{equation}
where  $V_{MP}(P)$ is the value of the perpetual minimum price guarantee contract,  $S_M(P)$  is the value of the project with a delayed perpetual minimum price guarantee contract (i.e., a minimum price guarantee contract that starts in the future moment $T$), and $ \frac{P Q}{r-\mu}e^{-(r-\mu)T}$ is the perpetual discounted profit from selling at market prices.

As a consequence of Proposition (\ref{propVC}), letting $C \to \infty$, it follows that the value of the perpetual minimum price guarantee contract, $V_{MP}$, is given by:
\begin{equation}
\label{eq:ProjectValueM}
V_{MP}(P) =  
\begin{cases}
L_1 P^{\beta_1} + \dfrac{F Q }{r} & \textrm{ for }  P < F  \\
\ \\
M_2 P^{\beta_2} + \dfrac{P Q}{r - \mu} & \textrm{ for }  P \geqslant F  
\end{cases} 
\end{equation}
where 
\begin{align}
\label{eq:A1_perpetual}
L_1 &= \dfrac{Q F^{1-\beta_1}} {\beta_1-\beta_2}\left(\frac{\beta_2}{r}-\frac{\beta_2 - 1}{r-\mu}\right)  >0 \\
\label{eq:B2_perpetual}
M_2 &= \dfrac{Q F^{1-\beta_2}} {\beta_1-\beta_2}\left(\frac{\beta_1}{r}-\frac{\beta_1 - 1}{r-\mu}\right) >0 
\end{align}
with $L_1$ is derived from \eqref{eq:E1} letting $C=\infty$, and $M_2=G_2$ is derived from \eqref{eq:G2}.  See Proposition (\ref{prop:ThresholdCollarEqualPriceFloor}) of Appendix \ref{app:proofs_relations}.

Moreover, if in (\ref{eq:SC}) we let again $C=\infty$, and using the properties of the normal distribution, we conclude that 
\begin{align}
\label{eq:SCM}
S_M(P) = & L_1 P^{\beta_1}\Phi(-d_{\beta_1}(P,F))+\dfrac{FQ}{r}e^{-rT}\Phi(-d_{0}(P,F))\nonumber  \\ 
& +M_2P^{\beta_2}\Phi(d_{\beta_2}(P,F))+\dfrac{PQ}{r-\mu}e^{-(r-\mu)T}\Phi(d_1(P,F))
\end{align}

Finally, as a consequence of Proposition \eqref{prop3}, and in view of these quantities, the next result holds. 

\begin{proposition}
The value of the option to invest with a minimum price guarantee design and the investment threshold are:
\begin{equation}
\label{eq:optionvalueM}
F_{M} (P) =  
\begin{cases}
(V_{M}(P_M^{\ast}) - I)\left(\dfrac{P}{P_M^{\ast}}\right)^{\beta_1} & \textrm{ for } P < P_M^{\ast} \\ \\
V_{M}(P) - I & \textrm{ for } P \geqslant P_M^{\ast}\end{cases}
\end{equation}  
where $P_M^\ast$ is the solution of 
\begin{equation}
\label{eq:trigger_floor}
\beta_1 (V_C(P_M^{\ast}) - I) - \dfrac{\partial V_C( P_M^{\ast} )}{\partial P_M^{\ast}}P_M^{\ast} = 0
\end{equation}
which in this case leads to two equations presented in Appendix \ref{app:floor}, that need to be solved using numerical methods.
\end{proposition}

We note that in this case the investment threshold $P_M^\ast$ is the solution of one of the two equations of \eqref{eq:trigger_floor}, depending on the relative ordering of $P_M^\ast$ and $F$. This is a direct consequence of \eqref{PiM}, as the instantaneous profit function depends also on the relative ordering of the price $P$ and the minimum price guarantee $F$, which impacts of the expression for the value of the project, as shown in \eqref{eq:ProjectValueM}. Therefore, when one is deriving the investment threshold $P_M^\ast$, we have two possible cases for the value matching and smoothing pasting conditions, depending on what instance of the value of the project we are using.

In a similar way to what we have considered in the previous sections, we also consider the case that the minimum price guarantee $F$ can be affected by a policy change before the project starts. As the reasoning is similar to the one presented previously, we omit further comments and present the results for a scenario with a finite minimum price guarantee contract and regulatory uncertainty. 

\begin{proposition}
The value of the option with a minimum price guarantee and regulatory uncertainty is given by:\footnote{The general option value function is $R_1 P^{\eta_1} + N_1^{(\omega)}P^{\beta_1}$.}
\begin{equation}
\label{eq:option_floor_reg}
F_{MR} (P) =  
\begin{cases}
(V_{M}(P_{MR}^{\ast}) - I - N_1^{(\omega)} {P_{MR}^\ast}^{\beta_1})\left(\dfrac{P}{P_{MR}^{\ast}}\right)^{\eta_{1}} + N_1^{(\omega)} P^{\beta_1}   & \textrm{ for } P < P_{MR}^{\ast}  \\
\ \\
V_M (P) - I & \textrm{ for } P \geqslant P_{MR}^{\ast}   
\end{cases}
\end{equation}
where
\begin{equation}
\label{eq:N1_omega}
N_1^{(\omega)} = \left(V_{M}^{(\omega)}\left(P_M^{\ast (\omega)} \right) - I \right)\left(\dfrac{1}{P_M^{\ast (\omega)}}\right)^{\beta_1}
\end{equation}

The investment threshold $P_{MR}^{\ast}$ is the solution of the following equation, that must be solved numerically:
\begin{equation}
\label{eq:trigger_floor_reg}
\eta_1(V_M(P_{MR}^{\ast}) - I - N_1^{(\omega)} {P_{MR}^\ast}^{\beta_1}) + \beta_1 N_1^{(\omega)} {P_{MR}^\ast}^{\beta_1} - \dfrac{\partial V_M(P_{MR}^{\ast} )}{\partial P_{MR}^{\ast}}P_{MR}^{\ast} = 0
\end{equation}
This equation results in three equations that are presented in Appendix \ref{app:floor}.
\end{proposition}

Given the relation between the sliding premium and the minimum price, one has the following dominance result. The proof can be found in Appendix \ref{app:proofs_relations}.

\begin{proposition}
\label{prop:VMPandVCP1}
The investment threshold of the project with a sliding premium  contract is always smaller than the investment threshold of the project with a minimum price guarantee contract.
\end{proposition}

\subsection{Fixed-price FIT}
\label{sec:Fixed FIT}

In this regime, the renewable energy producer has a fixed-tariff contract, where the firm receives $F$ for a period of time $T$. The fixed-price scheme is a special case of the collar regime when $C$ equal to $F$. Among all FIT designs, the fixed-price FIT is the most widely used FIT scheme around the world. For instance, fixed-price FITs have been offered in Germany, France, Portugal, Canada, and among other countries \cite{Couture10b}. 

The value of the project has a fixed-tariff contract with a finite duration $T$. After time $T$, the renewable energy producer sells the energy for the market price $P$.  Hence, after investment, the firm gets the following expected value in \eqref{eq:VF}:
\begin{equation}
\label{eq:VFfixed}
V_F(P)=E\left[\left. \int_{0}^{T} F Qe^{-rt}dt + \int_{T}^{+\infty} P_t Qe^{-rt}dt \right| P_0=P\right] = \dfrac{F Q }{ r}\left(1-e^{-rT}\right) +  \dfrac{P Q}{r - \mu}e^{-(r - \mu)T}.
\end{equation}
In Proposition \eqref{prop:ThresholdCollarEqualFixedPrice} of Appendix \ref{app:proofs_relations} we present a formal proof of such result, using the results from the sliding premium.

Therefore, in view of Proposition \eqref{prop3}, with $F=C$, it follows that
\begin{proposition}
\label{prop:OptionFixed}
The value of the investment option with a fixed price design is given by:\footnote{The general option value function is $S_1 P^{\beta_1}$.}
\begin{equation}
\label{eq:OptionFixed}
F_F (P) =  
\begin{cases}
\left(V_F(P_F^{\ast}) - I  \right) \left(\dfrac{P}{P_F^{\ast}}\right)^{\beta_1} & \textrm{ for }   P < P_F^{\ast} \\ \\
V_F(P) - I  & \textrm{ for }   P \geqslant P_F^{\ast}\end{cases}
\end{equation} 
where the investment threshold $P_F^{\ast}$ is equal to:
\begin{equation}
\label{eq:triggerfixed}
P_F^{\ast} = \dfrac{\beta_1}{(\beta_1-1)} \dfrac{r - \mu}{Q e^{-(r - \mu)T} } \left(I-\dfrac{F Q }{ r}\left(1 - e^{-rT}\right) \right).
\end{equation}
\end{proposition}

We note that in this case, contrary to the previous cases, we are able to find a closed expression for the investment threshold.

If regulatory uncertainty is present, then it follows from the previous results that the value function and investment threshold are given as follows.

\begin{proposition}
\label{prop:OptionFixedR}
The value of the investment option with a fixed price scheme and regulatory uncertainty is given by:\footnote{The general option value function is $U_1 P^{\eta_1} + S_1^{(\omega)}P^{\beta_1}$.}
\begin{equation}
\label{eq:OptionFixedR}
F_{FR} (P) =  
\begin{cases}
\left(V_F(P_{FR}^{\ast}) - I -S_1^{(\omega)}{P^\ast_{FR}}^{\beta_1}  \right) \left(\dfrac{P}{P_{FR}^{\ast}}\right)^{\eta_{1}}+ S_1^{(\omega)}P^{\beta_1} & \textrm{ for }   P < P_{FR}^{\ast} \\ \\
V_F(P) - I  & \textrm{ for }   P \geqslant P_{FR}^{\ast}
\end{cases}
\end{equation} 
where the investment threshold $P_{FR}^{\ast}$ is the solution of the following equation, which needs to be solved using numerical methods:
\begin{equation}
\label{eq:trigger_fixed}
-(\eta_1-\beta_1 )S_1^{(\omega)}{P^\ast_{FR}}^{\beta_1} + (\eta_1-1)\dfrac{P_{FR}^{\ast} Q}{r - \mu}e^{-(r-\mu)T}+\eta_1\left( \dfrac{F Q}{r}\left(1-e^{-rT}\right)-I\right) = 0.
\end{equation}
\end{proposition}

\subsection{Fixed-premium FIT}
\label{sec:Premium FIT}

Next, we analyze a fixed-premium contract, whereby the firm receives a bonus $F$ over the market price. Fixed-premium FIT policies have been offered in countries such as Spain, the Czech Republic and the Netherlands \cite{Couture10b}. In fact, Spain and  the Czech Republic let investors choose between a fixed-price FIT and a fixed-premium FIT. 

For this contract, the profit function is given as follows:

\begin{equation}
\Pi_P(P) = (P + F) Q 
\end{equation}
where $F$ is a premium over the electricity market price. It follows that this regime cannot be seen as a special case of a sliding premium with cap and floor.

If the firm invests with a current price $P$, the firm gets the following expected value:
\begin{equation}
\label{eq:}
V_P (P) = E\left[\left. \int_{0}^{T} ( P + F )Q e^{-rt}dt + \int_{T}^{+\infty} P Qe^{-rt}dt \right| P_0=P\right]
 = \dfrac{P Q }{r - \mu}+\dfrac{F Q}{r}\left(1-e^{-rT}\right).
\end{equation}
Using the results in Section \ref{sec:opt}, we obtain the following value of the investment option for this type of contract.

\begin{proposition}
The value of the investment option when the producer has a fixed-premium FIT contract is given by:\footnote{The general option value function is $V_1 P^{\beta_1}$.}
\begin{equation}
\label{eq:optionP}
F_P (P) =  
\begin{cases}
\left(V_P(P_P^{\ast}) - I  \right) \left(\dfrac{P}{P_P^{\ast}}\right)^{\beta_1} & \textrm{ for }   P < P_P^{\ast} \\ \\
V_P(P) - I  & \textrm{ for }   P \geqslant P_P^{\ast}\end{cases}
\end{equation} 
where the investment threshold $P_P^{\ast}$ is given by:
\begin{equation}
\label{eq:triggerP}
P_P^{\ast} = \dfrac{\beta_1}{\beta_1 - 1} \dfrac{r - \mu}{Q} \left(I-\dfrac{F Q }{ r}\left(1-e^{-rT}\right)\right).
\end{equation}
\end{proposition}

If we consider uncertainty in the retraction of the subsidy, a straightforward application of the results of Section \ref{sec:opt} lead to the following.

\begin{proposition}
The value of the investment option with a fixed-premium FIT scheme and regulatory uncertainty is given by:\footnote{The general option value function is $W_1 P^{\eta_1} + V_1^{(\omega)}P^{\beta_1}$.}
\begin{equation}
\label{eq:OptionPremiumR}
F_{PR} (P) =  
\begin{cases}
\left(V_F(P_{PR}^{\ast}) - I -V_1^{(\omega)}{P_{PR}^\ast}^{\beta_1}  \right) \left(\dfrac{P}{P_{PR}^{\ast}}\right)^{\eta_{1}}+ V_1^{(\omega)}P^{\beta_1} & \textrm{ for }   P < P_{PR}^{\ast} \\ \\
V_F(P) - I  & \textrm{ for }   P \geqslant P_{PR}^{\ast}
\end{cases}
\end{equation} 
where the investment threshold $P_{PR}^{\ast}$ is the solution of the following equation, which needs to be solved numerically.
\begin{equation}
\label{eq:trigger_premium}
-(\eta_1- \beta_1)V_1^{(\omega)}{P_{PR}^\ast}^{\beta_1} + (\eta_1-1)\dfrac{Q}{r - \mu}P_{PR}^{\ast}+\eta_1\left(\dfrac{F Q}{r}\left(1-e^{-rT}\right)-I \right) = 0.
\end{equation}
\end{proposition}

\section{Analytical and Numerical Study}
\label{sec:Comparative statics analysis Trigger}
In this section, we study the influence of some parameters on the investment thresholds of the FIT schemes addressed in this paper. In some cases, we are able to present analytical results and corresponding plots of the thresholds; whereas in other cases, we only present the plots based on numerical results for a given set of parameter values. For the latter case,  we have performed many other tests with other parameter values, besides the one presented in this section, and the results are qualitatively the same.

For the numerical study,  we use a typical European onshore wind farm with 25 wind turbines \cite{enevoldsen16},  although our FIT models do not have any technology-specific characteristics. The investment cots of each turbine is 1.5 Million Euros / MW \cite{EWEA09} and each turbine has a 2MW capacity. We also assume that the wind turbine's capacity factor is 30\%, which is a reasonable estimate according to \citeasnoun{EWEA09}.

In order to simplify the analysis, we calculate the investment threshold for a single turbine but the results can be easily extended for any number of turbines. We also assume that all parameters have annualized values. For instance, a fixed-price FIT with an $F$ equal to \EUR{25} / MWh generates an annual revenue of \EUR{131,400.00} (i.e., 30 \% x 2 MWh x \EUR{25} / MWh x 24 hours x 365 days).

In addition, we use the same values of \citeasnoun{Ritzenhofen16} for the GBM parameters, which are based on real data. Table \ref{tab:parameters} summarizes the base-case parameters of the numerical study in this section. In particular for the sliding premium, we assume that the annual revenue from the price cap is equal to \EUR{300,000.00} ($\approx$ \EUR{57.08} / MWh).

\begin{table}[!ht]
\caption{Base-case parameters used to calculate the thresholds}
\label{tab:parameters}
%\footnotesize
\centering
    \begin{tabular}{llll|}
  \hline \hline
  $r$ & risk-free rate & 5\%\\  
  $F$ & tariff & \EUR{25} / MWh \\
%  $C$ & price cap & \EUR{58} / MWh \\
  $T$ & finite duration of FIT & 15 years \\
  $\mu$ & deterministic drift & 0\%\\
  $\sigma$ & volatility & 19\% \\
  $I$ & total investment cost & \EUR{3} Millions\\
  $\omega$ & the reduction of F is (1 - $\omega$) & 80\% \\
  $\omega_C$ & the reduction of $C$ is (1 - $\omega_C$) & 100\% \\
  $\lambda$ & mean arrival rate of a jump event & 0.5 \\
  \hline \hline
\label{tab:parametersPT}
\end{tabular}\\
\end{table}

Next, we present the comparative statics results for the two parameters related with regulatory uncertainty, that we are able to derive analytically, by manipulation of the equations that define implicitly the investment thresholds (as in the case of regulatory uncertainty, the investment thresholds for all the regimes are defined implicitly, by some non-linear equation). The proofs are presented in Appendix \ref{app:proof_comparative_statics}, and are somehow simplified, as they follow from long calculations.

\begin{proposition}
\label{prop:comparative_statics}
When there is regulatory uncertainty, the investment thresholds for all the four schemes considered in this paper decrease with $\lambda$  and increase with $\omega$.
\end{proposition}

Therefore, investment is anticipated when it is more likely to occur a reduction of the tariff, because investors prefer to anticipate the investment decision in order to obtain a higher tariff for a longer period. 
Figure \ref{fig:TriggerLambda} illustrates this behavior, for different values of the parameter $\lambda$. We also include the investment threshold in a free-market condition (i.e., $P^{\ast}_W$),\footnote{We calculate $P^{\ast}_W$ with the threshold of the fixed-premium where the tariff $F$ is equal to 0.} where investors do not have any FIT policies available.  Note that, for the sliding premium, we do not consider changes in the price cap  (i.e., $\omega_C = 1$) due to regulatory uncertainty. We analyze the effect of regulatory uncertainty on the price cap later in this section.

\begin{figure}[!ht]
    \centering
    \includegraphics[width=0.7\textwidth]{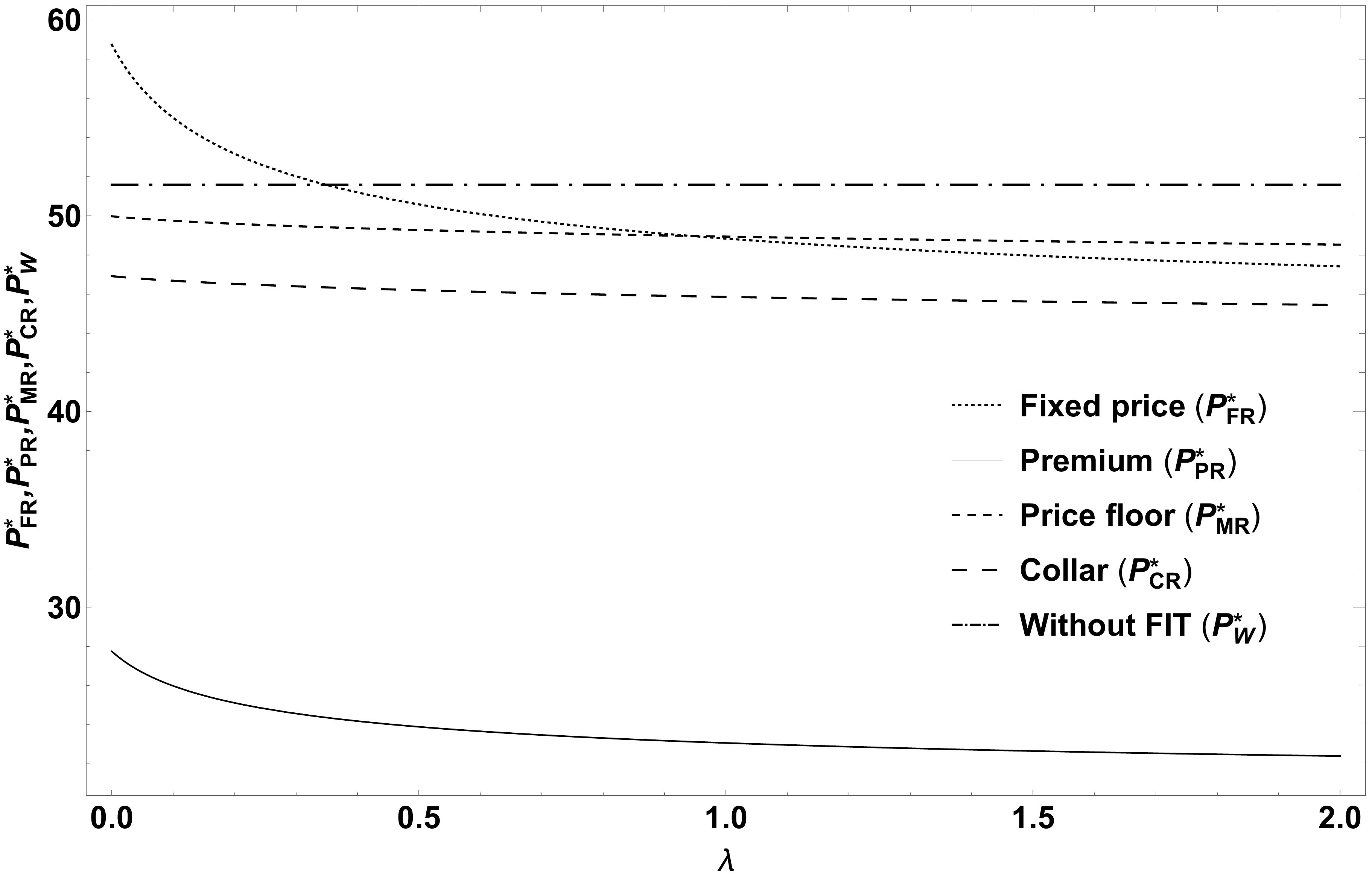}
    \caption{Triggers $P_{FR}^{\ast}$, $P_{PR}^{\ast}$, $P_{MR}^{\ast}$ and $P_{CR}^{\ast}$ as a function of $\lambda$}
    \label{fig:TriggerLambda}
\end{figure}

An interesting result in Figure \ref{fig:TriggerLambda} is that increasing $\lambda$ generates a higher reduction in the fixed price and fixed premium thresholds than in the price-floor and collar thresholds. The intuition behind this effect is due to the reduction and how it affects the four FIT policies. While the reduction in the fixed-price and premium-price policies always leads to a revenue loss, the price-floor and collar regimes only suffer a revenue loss when $P < F$.

A similar effect occurs if the reduction of the tariff is larger.  Figure \ref{fig:TriggerOmega} illustrates such behavior.  The plot of the investment threshold of the collar regime assumes that the reduction only affects the price floor. From Figure \ref{fig:TriggerOmega},  we can observe that all the investment thresholds decrease as $\omega$ decreases. This is due to the fact that lower values of $\omega$ produce higher reductions in the tariff, and consequently a lower expected profit. These potential losses generate lower investment thresholds because investors decide to accelerate investment in order to obtain a higher tariff before the jump event occurs. In addition, the effect of $\omega$ is greater in the fixed-price and fixed-premium FIT schemes, as shown in Figure \ref{fig:TriggerOmega}, because of the same reason explained for the parameter $\lambda$.

\begin{figure}[!ht]
    \centering
    \includegraphics[width=0.7\textwidth]{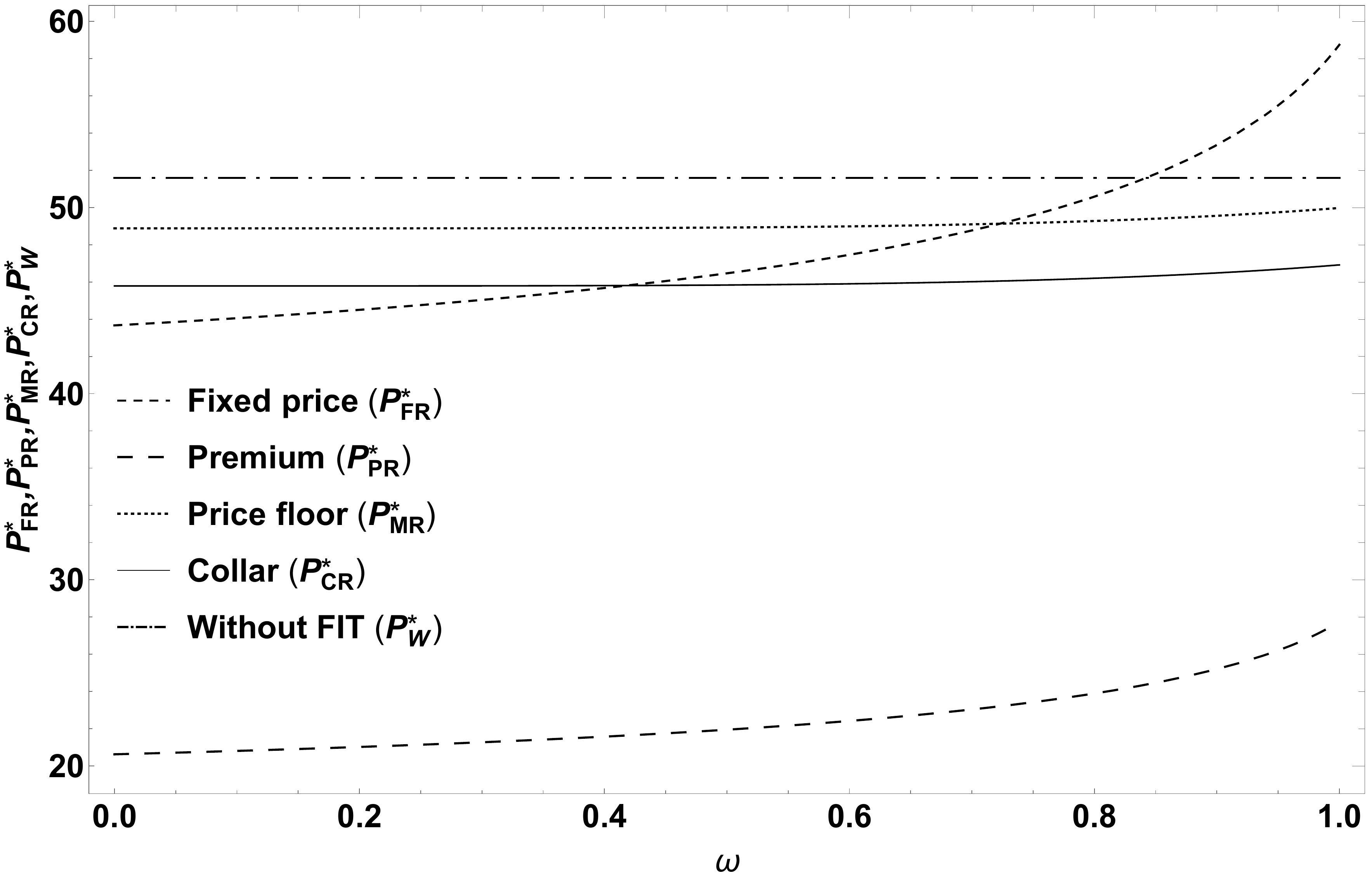}
    \caption{Triggers $P_{FR}^{\ast}$, $P_{PR}^{\ast}$, $P_{MR}^{\ast}$ and $P_{CR}^{\ast}$ as a function of $\omega$}
    \label{fig:TriggerOmega}
\end{figure}

Although we assume in Figure \ref{fig:TriggerOmega} that a policy-making decision could only reduce the price floor in the collar regime, we know that policy uncertainty may also affect the price cap. Hence, Figures \ref{fig:TriggerColarOmegaC} and \ref{fig:TriggerColarLambdaOmegaC} present the investment threshold when a policy reduction affects the price cap. We assume that the price cap may be reduced to $\omega_C C$ when a jump event occurs with a probability $\lambda dt$. In Figure \ref{fig:TriggerColarOmegaC}, we observe that the investment threshold decreases as $\omega_C$ decreases, thus accelerating the investment decision. In addition, Figure \ref{fig:TriggerColarLambdaOmegaC}  plots the investment threshold when $\omega =1$ and $\omega_C =0.8$, and we can see that the threshold reduces as $\lambda$ increases. Hence, a higher probability of occurring a reduction makes investors accelerate the decision-making process in order to obtain a higher tariff before the jump event.

Another interesting observation occurs when we compare the collar's threshold plot in Figure \ref{fig:TriggerLambda}, where $\omega = 0.8$ and $\omega_C = 1$, with Figure \ref{fig:TriggerColarLambdaOmegaC}, where $\omega = 1$ and $\omega_C = 0.8$. As $\lambda$ increases, we observe that the investment threshold has a larger reduction when the policymaker reduces the price cap, as observed in Figure \ref{fig:TriggerColarLambdaOmegaC}, than the price floor, as shown in Figure \ref{fig:TriggerLambda}. In a similar comparison, we observe a larger reduction in the investment threshold in Figure \ref{fig:TriggerColarOmegaC}, where $\omega_C$ varies in the interval $[0,1]$, than the reduction  in the collar's threshold plot in Figure \ref{fig:TriggerOmega}, where $\omega$ varies in the interval $[0,1]$. In summary, these plots show that the price cap reduction has a higher impact on investment threshold than the price floor reduction.

\begin{figure}[!ht]
\centering
\begin{subfigure}{.5\textwidth}
  \centering
  \includegraphics[width=0.97\linewidth]{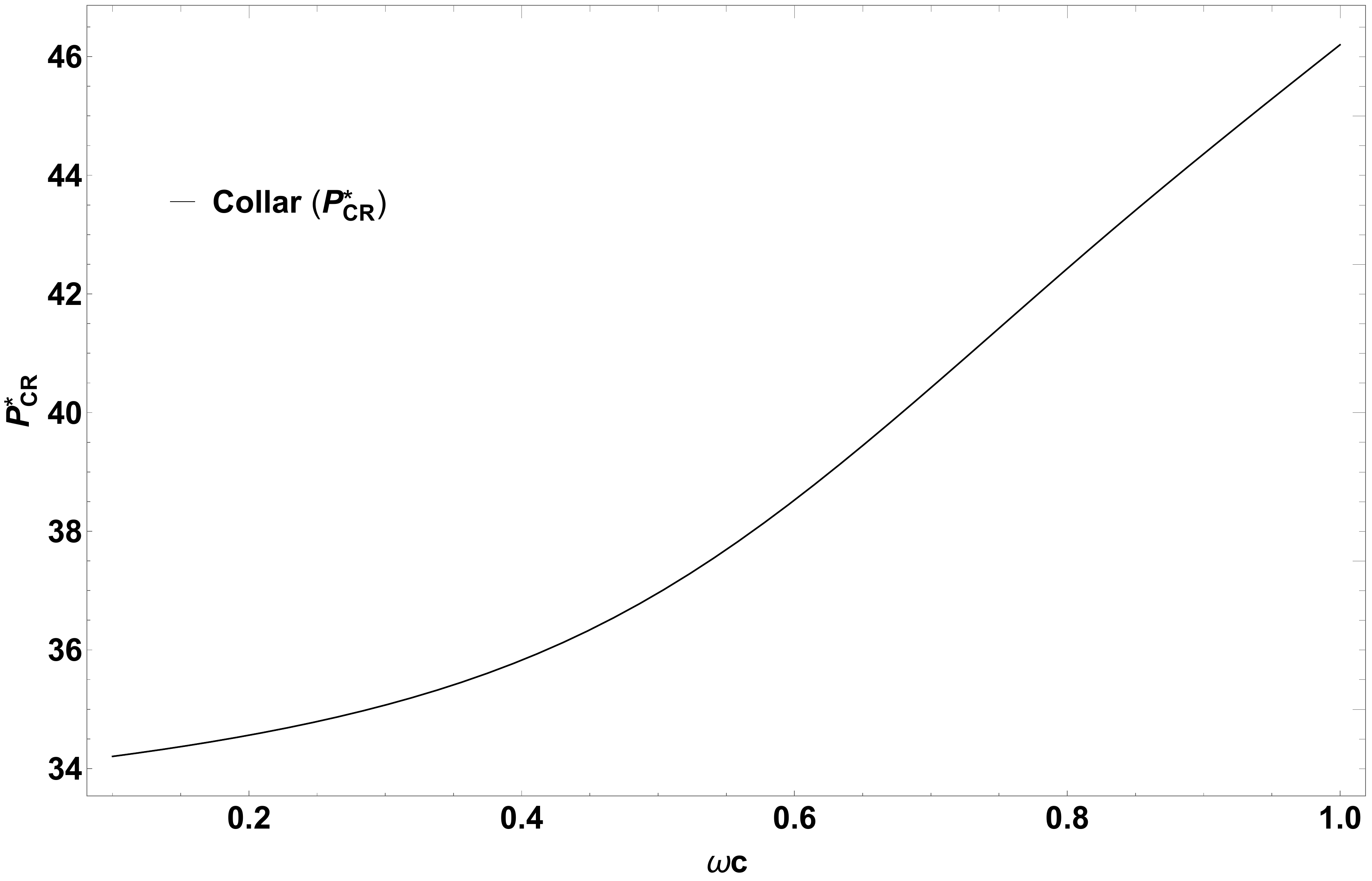}
  \caption{Trigger as a function of $\omega_C$}
  \label{fig:TriggerColarOmegaC}
\end{subfigure}%
\begin{subfigure}{.5\textwidth}
  \centering
  \includegraphics[width=0.98\linewidth]{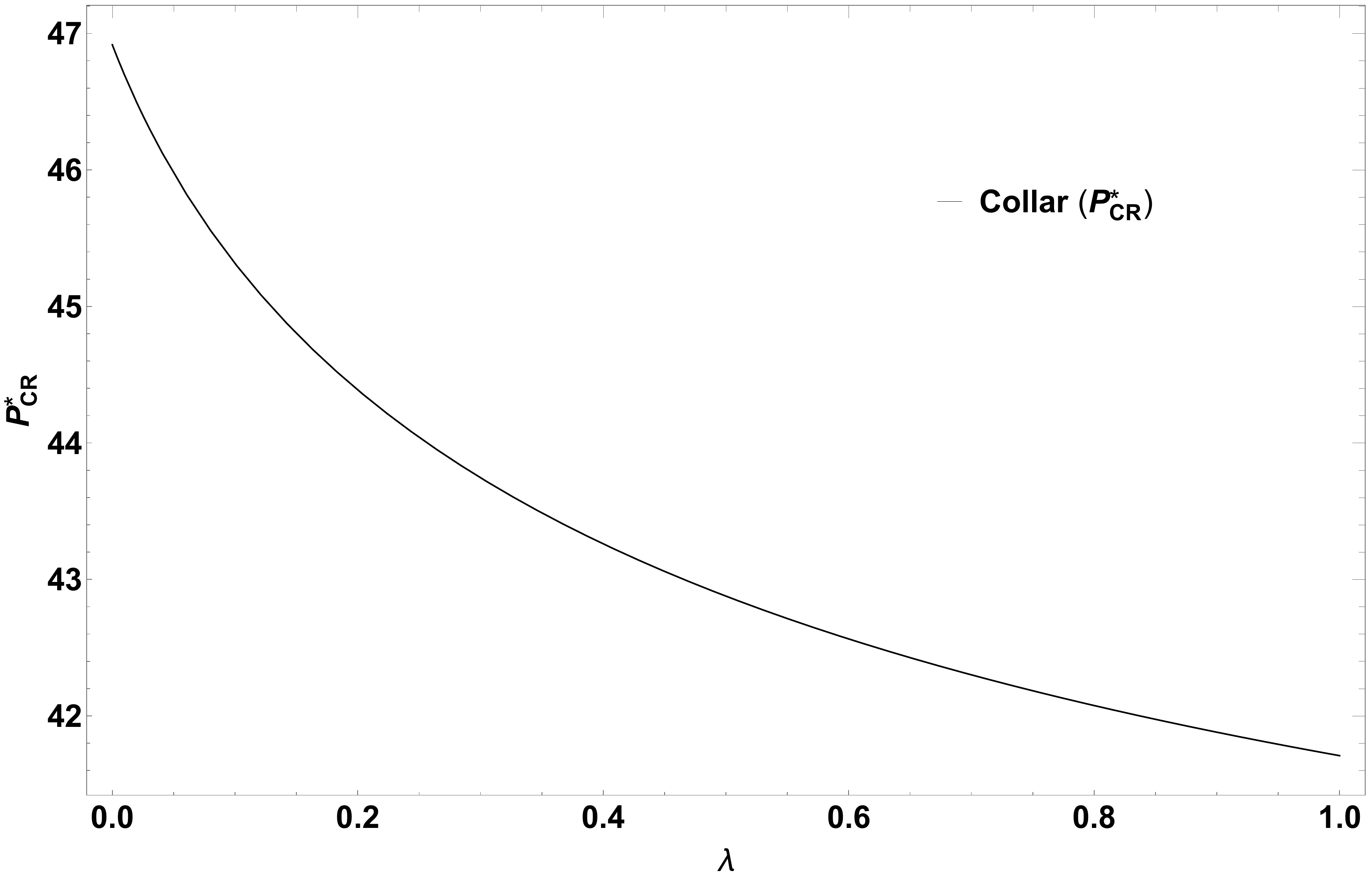}
  \caption{Trigger as a function of $\lambda$}
  \label{fig:TriggerColarLambdaOmegaC}
\end{subfigure}
\caption{Trigger $P_{CR}^{\ast}$ as a function of $\omega_C$ and $\lambda$}
\label{fig:test}
\end{figure}

\begin{figure}[!ht]
    \centering
    \includegraphics[width=0.7\textwidth]{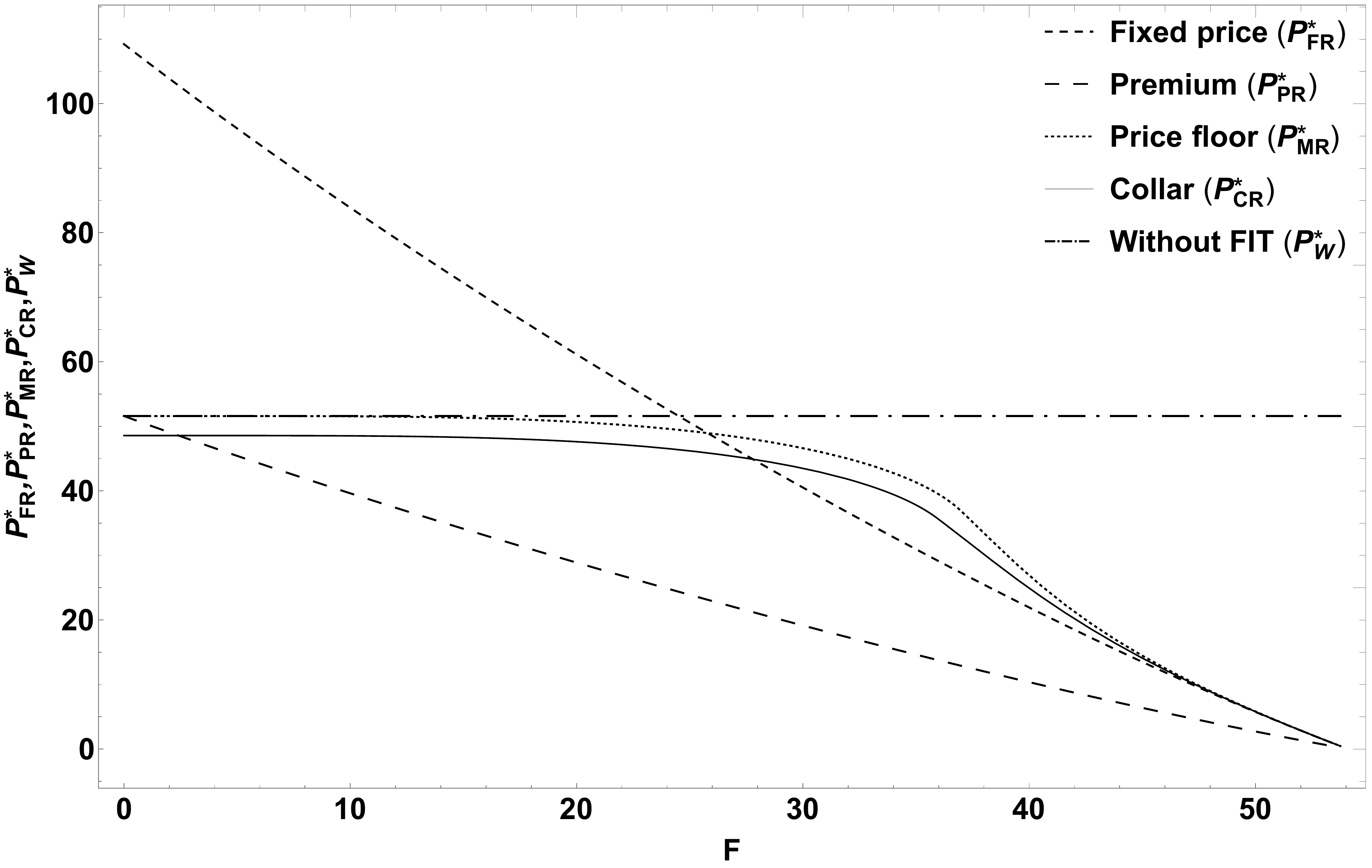}
    \caption{Thresholds $P_{FR}^{\ast}$, $P_{PR}^{\ast}$, $P_{MR}^{\ast}$ and $P_{CR}^{\ast}$ as a function of $F$}
    \label{fig:TriggerF}
\end{figure}

Figure \ref{fig:TriggerF} presents the investment thresholds for the four different FIT policies as a function of the tariff $F$. The plots show that the investment thresholds of the FIT policies decrease as $F$ increases; in other words, the decision to invest is accelerated when $F$ increases\footnote{The same result holds for a scenario without regulatory uncertainty}. In the Appendix we find the analytical proof of such result for the fixed price and fixed premium (with and without regulatory uncertainty).

We also observe that the thresholds for the fixed-premium regime, price-floor regime, and collar regime are lower than the investment threshold in a free-market condition. In fact, these are expected results for policymakers because a key goal of a price-based incentive, such as a FIT, should be to accelerate the investment decision. The only exception is the fixed-price regime for values of $F$ approximately lower than \EUR{25}/MWh, where the investment threshold in a free-market condition is below the investment threshold of the fixed-price regime. This result suggests that investors prefer the free-market condition to a low remuneration from the fixed-price FIT. Hence, policymakers should offer values of the tariff with a fixed-price scheme greater than this point. In other words, this is the minimum value of the tariff that policymakers should offer within a fixed-price FIT contract.

The plots of the investment threshold of the collar regime and price-floor regime in Figure \ref{fig:TriggerF} present two relevant results for policymakers. First, the price floor regime is a special case of the collar regime as shown in Figure \ref{fig:Triggercollarpricefloor}, whereby the investment threshold of the collar regime rapidly converges to the price-floor regime as we increase the value of the cap $C$.\footnote{In Appendix \ref{app:proofs_relations} at Proposition \ref{prop:ThresholdCollarEqualPriceFloor} we prove that the value of the project with a collar regime converges to a price-floor regime when $C \to +\infty$. Consequently, the thresholds of both regimes are the same when $C \to +\infty$.} Second, we observe from Figure \ref{fig:TriggerF} that the collar threshold is below the price-floor threshold when both regimes have the same price floor value. This result of accelerating the investment decision is due to the cap, especially for lower values of the price floor. An explanation for this result is that investors want to avoid receiving the cap, and hence prefer to start the investment when the market price is lower. From a policy-making perspective, this result suggests that the collar regime is a better policy than the price-floor regime because it accelerates the investment while  avoiding excessive earnings to producers.
In addition, we analytically prove that the value of the project with a price-floor regime is always greater than the value of the project with a collar regime in Appendix \ref{app:proofs_relations} at Proposition \ref{prop:VMPandVCP1}. These are expected results because an investor should have a higher profit without the price cap.

\begin{figure}[!ht]
    \centering
    \includegraphics[width=0.7\textwidth]{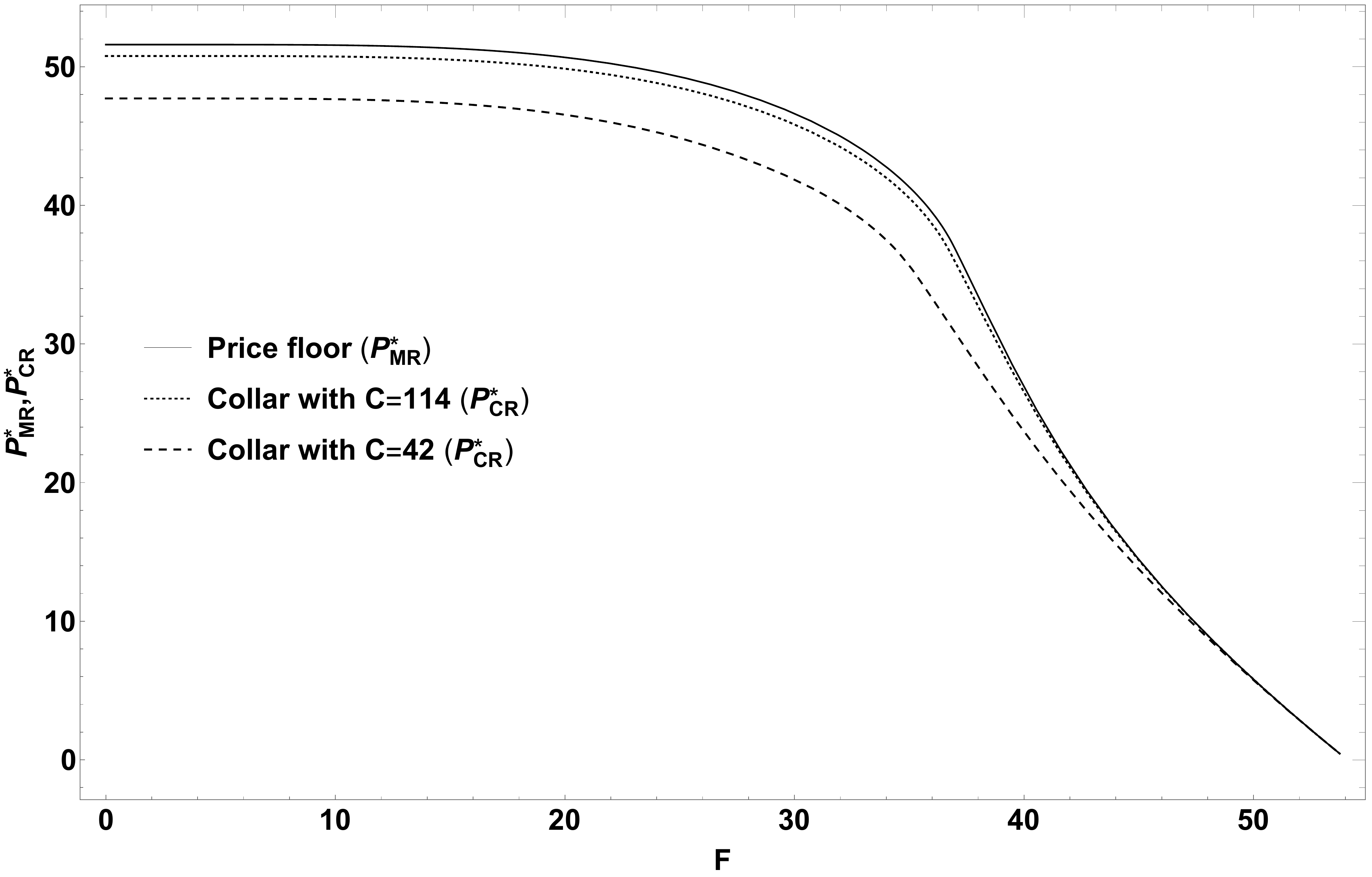}
    \caption{Triggers $P_{MR}^{\ast}$ and $P_{CR}^{\ast}$ as a function of the tariff $F$ for different values of the cap $C$}
    \label{fig:Triggercollarpricefloor}
\end{figure}

Figure \ref{fig:Triggerpc} shows a plot of the investment thresholds of the collar scheme as a function of the cap $C$. From the figure, we can observe that the investment threshold has a minimum point where $C$ is approximately equal to 42. Therefore, policymakers should not offer caps above this value.

\begin{figure}[!ht]
    \centering
    \includegraphics[width=0.6\textwidth]{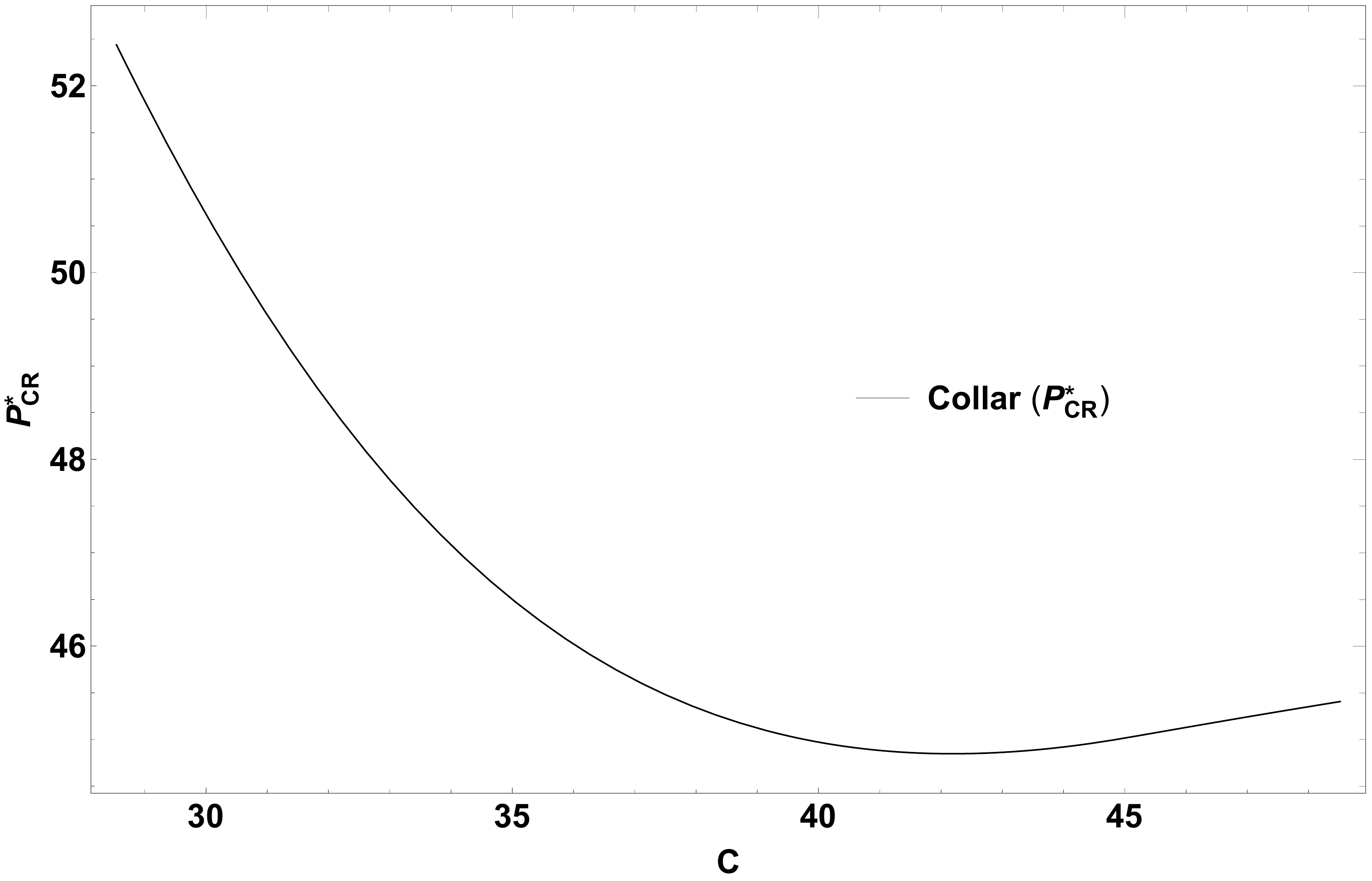}
    \caption{Threshold $P_{CR}^{\ast}$ as a function of the price cap $C$}
    \label{fig:Triggerpc}
\end{figure}

Figure \ref{fig:TriggerSigma} presents the plots of the investment thresholds for different values of the volatility $\sigma$. The results are consistent with the real options theory where higher volatilities increase the thresholds and consequently postpone the investment decision. Moreover, in the Appendix we prove analytically such result for the fixed price and for the fixed premium regimes, with and without regulatory uncertainty.

\begin{figure}[!ht]
    \centering
    \includegraphics[width=0.7\textwidth]{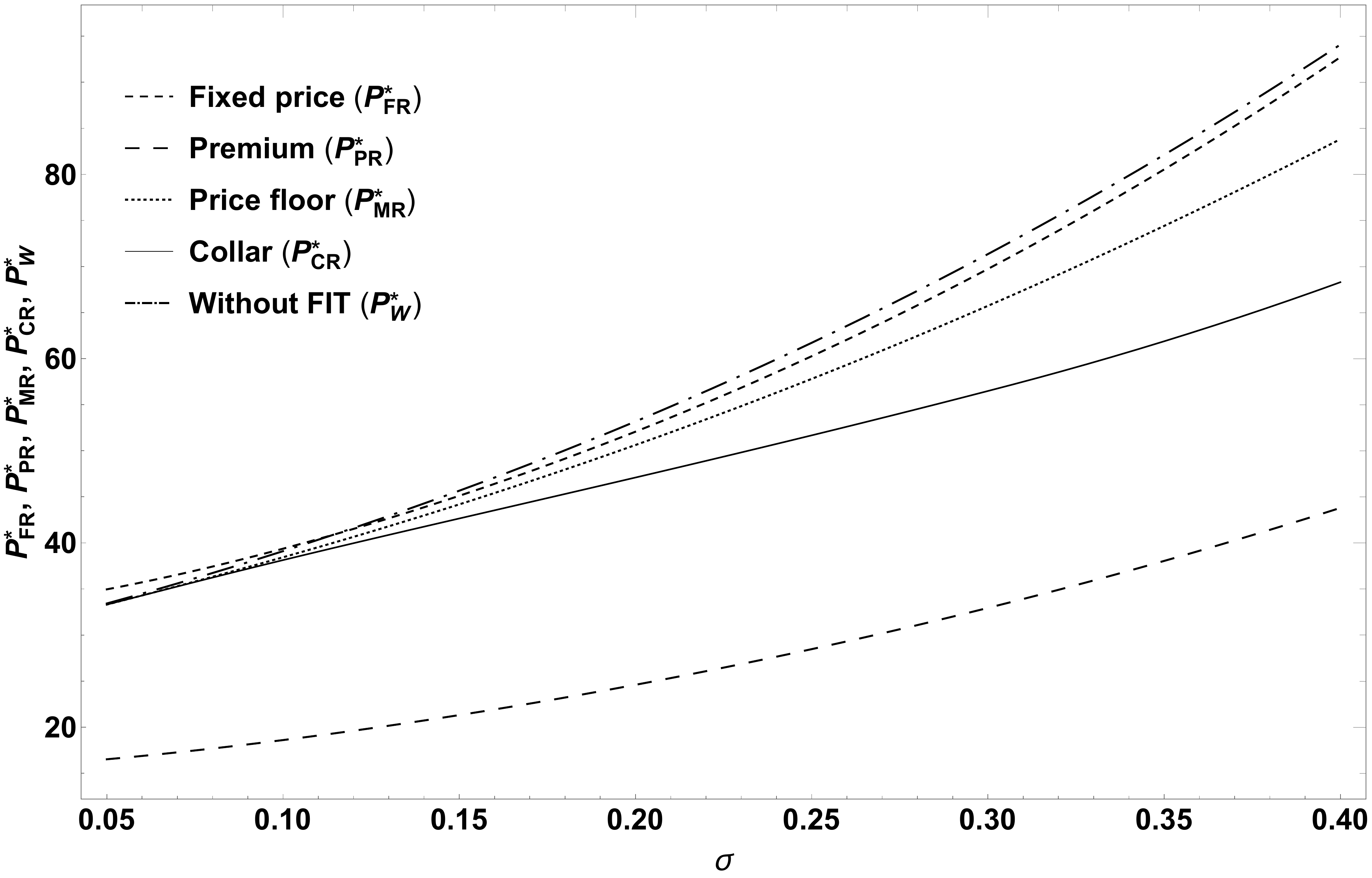}
    \caption{Thresholds $P_{FR}^{\ast}$, $P_{PR}^{\ast}$, $P_{MR}^{\ast}$ and $P_{CR}^{\ast}$ as a function of $\sigma$}
    \label{fig:TriggerSigma}
\end{figure}

In Figure \ref{fig:TriggerT_all} we present the plots of the investment thresholds for different values of the duration of the contract $T$. The premium, floor and collar regimes presented in Figure \ref{fig:TriggerT} accelerate the investment when policymakers increase the duration of the contract. However, the fixed-price regime is different when the tariff is low, as shown in Figure \ref{fig:TiggerFixedT}. For a tariff equal to \EUR{25}/MWh, note that the fixed-price FIT duration has a non-monotonic effect: it initially accelerates investment and then it  defers investment as we increase the duration of the contract. Figure \ref{fig:TiggerFixedT} also shows that this effect does not occur for higher tariffs (i.e., \EUR{37.5}/MWh and \EUR{50}/MWh) because the investment is accelerated when the duration of the contract increases. Hence, policymakers should not use low tariffs in the fixed-price regime because it might generate an undesired behavior of postponing the investment decision. 

\begin{figure}[!ht]
\centering
\begin{subfigure}{.5\textwidth}
  \centering
  \includegraphics[width=0.98\linewidth]{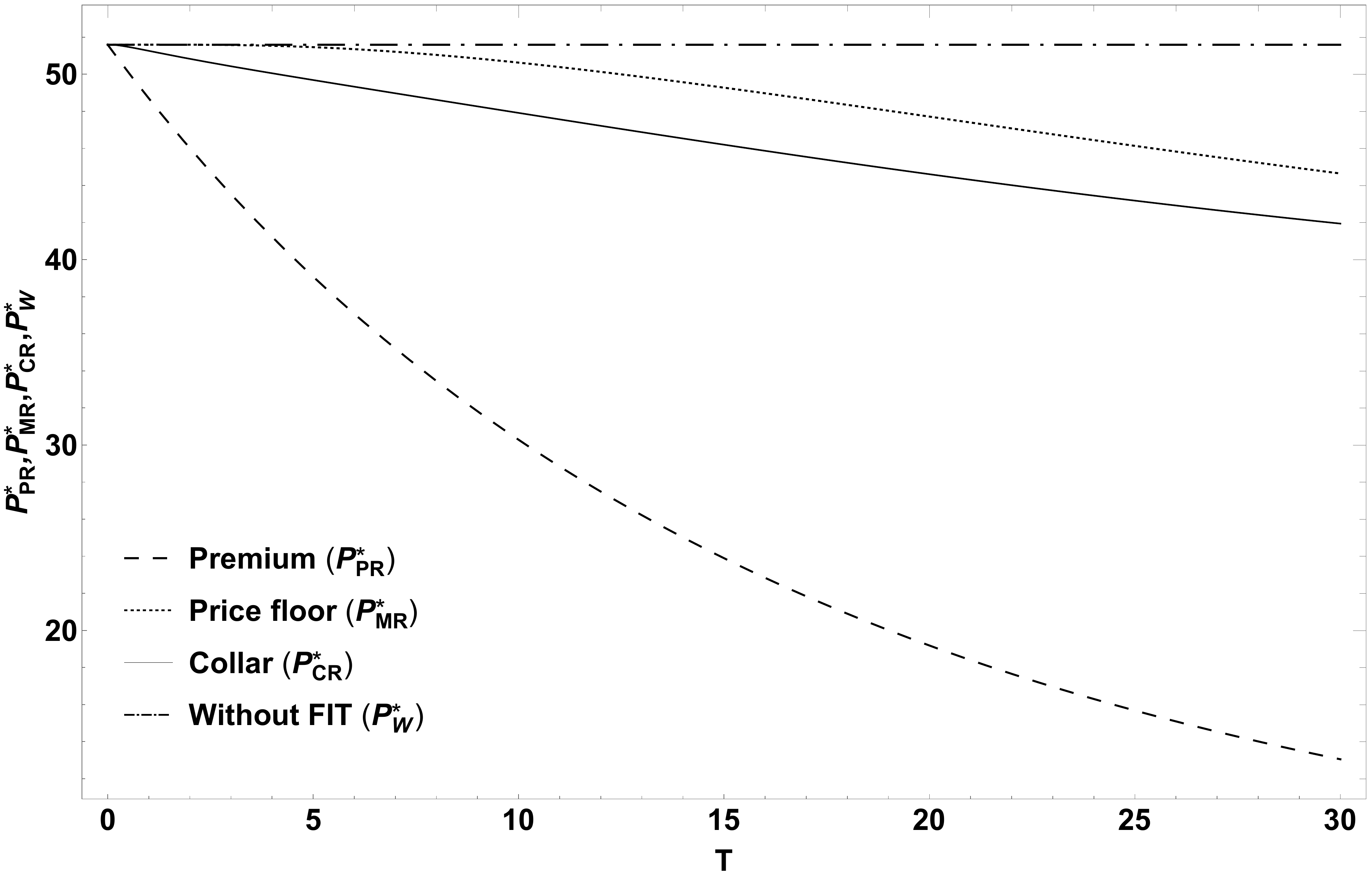}
  \caption{Premium, Floor and Collar regimes}
  \label{fig:TriggerT}
\end{subfigure}%
\begin{subfigure}{.5\textwidth}
  \centering
  \includegraphics[width=0.98\linewidth]{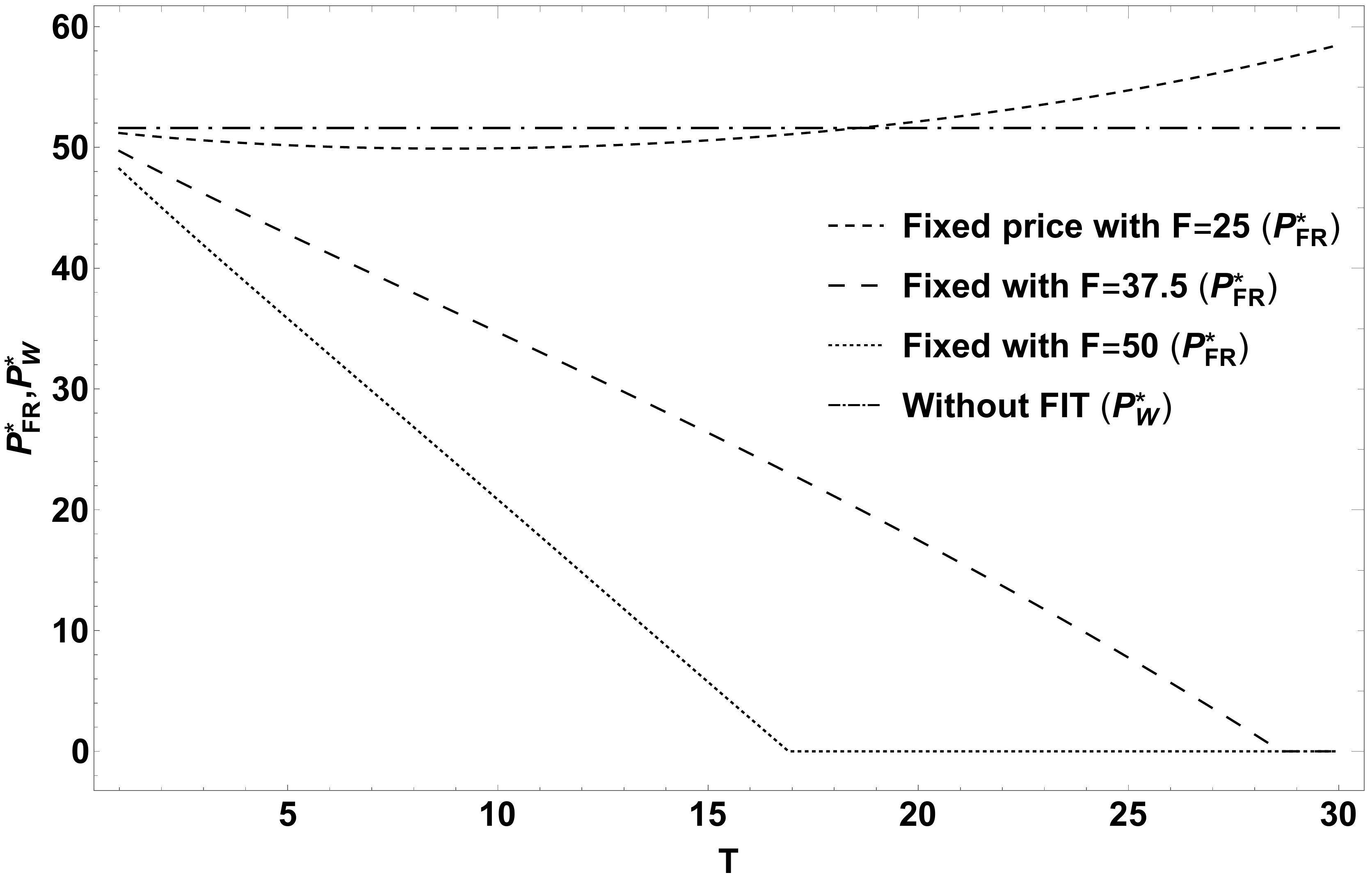}
  \caption{Fixed price regime for different values of $F$}
  \label{fig:TiggerFixedT}
\end{subfigure}
\caption{Investment threshold as function of $T$}
\label{fig:TriggerT_all}
\end{figure}

From the previous figures in this section, one might draw a conclusion that the fixed-premium regime is the best FIT contract because it presents the lowest investment trigger. However, this is a misleading conclusion because the previous plots consider that all regimes have the same tariff value $F$. In particular, the fixed-premium regime considers that the remuneration is equal to the market price plus the tariff $F$, which generates a higher remuneration than all the other regimes. Figure \ref{fig:FxTrigger} presents the values of the tariff that generate the same investment trigger. As expected, the fixed-premium plot is below all the other plots because a policymaker should use a smaller tariff for this regime in order to generate the same threshold as the other regimes. Figure \ref{fig:FxTrigger} can also be used as a tool for policy-making because it presents the values of the tariff $F$ that generate the same trigger for the four FIT policies.

\begin{figure}[!ht]
    \centering
    \includegraphics[width=0.8\textwidth]{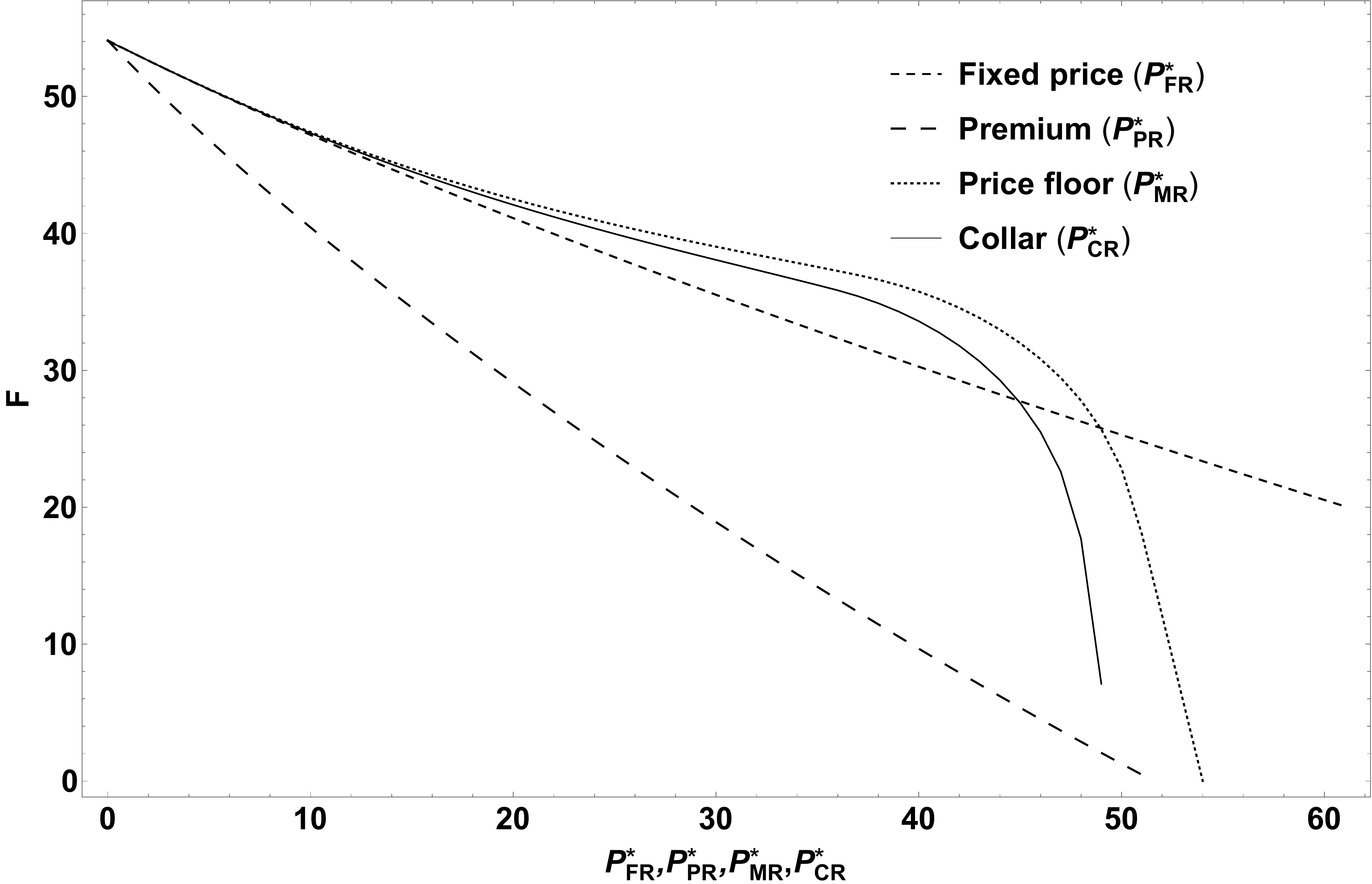}
    \caption{Feed-in tariff $F$ as a function of the thresholds}
    \label{fig:FxTrigger}
\end{figure}

\section{Concluding remarks}
\label{sec:conclusions}

This work analyzes two different feed-in premiums, namely a minimum price guarantee and a sliding premium with cap and floor. We also include the fixed-price FIT and fixed-premium FIT in the analysis, because these schemes have been widely analyzed within the scientific community. For each scheme, we use  a semi-analytical real options framework in order to calculate the value of the project, the option value, and the optimal investment threshold. We also extend the models of each FIT scheme in order to include regulatory uncertainty, where we use a Poisson process to model an occasional reduction on the tariff before the project starts.

Regarding the effect of the regulatory uncertainty, the results show that a higher probability of occurring a reduction in the tariff accelerates the investment because investors are anticipating the investment decision in order to obtain a higher tariff. In addition, a higher probability of occurring a policy change generates a higher reduction in the fixed-price and fixed-premium thresholds than the thresholds of the price-floor regime and the collar regime. This effect is due to a revenue loss that always occurs in the fixed-price and fixed-premium schemes in the event of a policy change, while the price-floor and collar schemes only suffer a revenue loss when the market price is below the guarantee. Regarding the analysis of the tariff reduction, we also reach the same conclusions. Thus, a higher reduction of the tariff accelerates the investment. We also find that the effect of the tariff reduction is greater for the fixed-price and fixed-premium thresholds than the price-floor and collar thresholds.

This paper also presents several results for policy-making. First, the sliding premium  has an interesting property. The possibility of reducing only the cap leads to a lower threshold when compared with the same effect on the price floor. Hence, the collar regime under the risk of regulatory uncertainty in the price cap has a higher impact on the investment threshold than the price floor.

Second, the thresholds of the fixed-premium regime, price-floor regime, and collar regime are lower than the investment threshold in a free-market condition. In fact, this is an expected result, because these policy schemes are aimed at accelerating a renewable energy investment. However, the fixed-price regime only has a threshold below the free-market condition when the tariff is greater than a certain value. Consequently, policymakers should avoid offering low values of the tariff in the fixed-price regime. 

Third, the results show that the collar regime has a lower investment threshold than the price-floor regime, especially when the price floor assumes lower values. This could suggest that the collar regime is a better policy because it accelerates investment and has the property of avoiding excessive earnings to producers. 

Fourth, increasing the duration of the contract reduces the investment threshold of the fixed-premium regime, price-floor regime, and collar regime. However, this effect does not occur in the fixed-price regime if the tariff is low. Lastly, we present an interesting tool for policy-making with our model, whereby policymakers can find the values of the tariff $F$ that generate the same trigger for the four FIT policies.

The current paper can be extended in several ways. From a policy-making perspective, the welfare effects of the different FIT schemes are relevant. These include not only the budget costs but also the effects on the consumers' and producers' surpluses, as well as the environmental effects. Further research could also consider the optimal capacity choice of the renewable energy producer, following \protect \citeasnoun{bar1999timing} and \protect \citeasnoun{dangl1999investment} for the monopoly case, and \protect \citeasnoun{huisman2015strategic} for the duopoly case, and \protect \citeasnoun{Boomsma12} and \protect \citeasnoun{Chronopoulos16} for price-taker firms. Another interesting extension is the possibility of retroactive reduction of the tariff (after the investment occurs). Since we deal with finite-lived schemes, a numerical method is required to find the project value, as it becomes time-dependent (the value of the project depends upon when the FIT is reduced).

\section*{Acknowledgements}

This work was supported by a PhD scholarship provided by Funda\c{c}\~{a}o para a Ci\^{e}ncia e a Tecnologia (FCT) through the MIT Portugal program under project SFRH/BD/52086/2013, and was supported by national funds through Funda\c{c}\~{a}o para a Ci\^{e}ncia e a Tecnologia (FCT) with reference UID/CEC/50021/2019. In addition, this work was carried out within the funding with COMPETE reference n\textsuperscript{o} POCI-01-0145-FEDER-006683 (UID/ECO/03182/2013), with the FCT/MEC's (Funda\c{c}\~{a}o para a Ci\^{e}ncia e a Tecnologia, I.P.) financial support through national funding and by the ERDF through the Operational Programme on ``Competitiveness and Internationalization - COMPETE 2020'' under the PT2020 Partnership Agreement. 
Cl\'audia Nunes acknowledge support from FCT through project reference  \text{FARO\_PTDC}/EGE-ECO/30535/2017 and the Research Council of Norway through project nr. 268093.

\newpage

\begin{appendix}

\section{Equations for finding the triggers}

In this section we present the equations that need to be solved numerically to obtain the triggers for investment for the sliding premium and minimum price guarantee FITs.

\subsection{Sliding premium with cap and floor}
\label{app:collar}

For the case without regulatory uncertainty, Equation \eqref{eq:trigger_collar} results in following three equations:
\begin{equation}
\label{eq:trigger_collar_app}
\begin{cases}
- (\beta_1 - \beta_2) \left( G_2 \left( \Phi(d_{\beta_2}(P_C^{\ast},F)) - \Phi(d_{\beta_2}(P_C^{\ast},C)) \right) \right. \\
\qquad \qquad \qquad \left. + H_2 \Phi(d_{\beta_2}(P_C^{\ast},C)) \right) {P_C^{\ast}}^{\beta_2} \\
\qquad + (\beta_1 - 1)\dfrac{ P_C^\ast Q}{r-\mu}e^{-(r-\mu)T} \left(1- \left( \Phi(d_1(P_C^{\ast},F)) - \Phi(d_1(P_C^{\ast},C)) \right) \right) \\
\qquad + \beta_1 \dfrac{F Q}{r}\left( 1-e^{-rT} \left(1- \Phi(d_0(P_C^{\ast},F)) \right) \right) \\
\qquad - \beta_1 \left(\dfrac{C Q}{r} e^{-rT} \left(1- \Phi(d_0(P_C^{\ast},C)) \right)+ I \right) = 0 & \textrm{ for } P_C^\ast < F \\
\\
(\beta_1 - \beta_2) \left( G_2 \left(1 - \left( \Phi(d_{\beta_2}(P_C^{\ast},F)) - \Phi(d_{\beta_2}(P_C^{\ast},C)) \right) \right)  \right. \\
\qquad \qquad \qquad \left. - H_2 \Phi(d_{\beta_2}(P_C^{\ast},C)) \right) {P_C^{\ast}}^{\beta_2} \\
\qquad + (\beta_1 - 1)\dfrac{ P_C^\ast Q}{r-\mu} \left(1+ e^{-(r-\mu)T} \left(1 - \left( \Phi(d_1(P_C^{\ast},F)) - \Phi(d_1(P_C^{\ast},C)) \right) \right) \right) \\
\qquad - \beta_1 \dfrac{F Q}{r}\left( 1-e^{-rT} \left(1- \Phi(d_0(P_C^{\ast},F)) \right) \right) \\
\qquad - \beta_1 \left(\dfrac{C Q}{r} e^{-rT}  \Phi(d_0(P_C^{\ast},C))+ I \right) = 0 & \textrm{ for } F \leqslant P_C^\ast < C \\
\\
- (\beta_1 - \beta_2) \left( G_2 \left( \Phi(d_{\beta_2}(P_C^{\ast},F)) - \Phi(d_{\beta_2}(P_C^{\ast},C)) \right)  \right. \\
\qquad \qquad \qquad \left. - H_2 \left( 1 - \Phi(d_{\beta_2}(P_C^{\ast},C)) \right)\right) {P_C^{\ast}}^{\beta_2} \\
\qquad + (\beta_1 - 1)\dfrac{ P_C^\ast Q}{r-\mu}e^{-(r-\mu)T} \left(1- \left( \Phi(d_1(P_C^{\ast},F)) - \Phi(d_1(P_C^{\ast},C)) \right) \right) \\
\qquad - \beta_1 \dfrac{F Q}{r} e^{-rT} \left(1- \Phi(d_0(P_C^{\ast},F)) \right) \\
\qquad + \beta_1 \left(\dfrac{C Q}{r} \left(1- e^{-rT} \left(1- \Phi(d_0(P_C^{\ast},C)) \right) \right)- I \right) = 0 & \textrm{ for } P_C^\ast \geqslant C
\end{cases}
\end{equation}

For the case with regulatory uncertainty, Equation \eqref{eq:trigger_collar_reg} results in following three equations:
\begin{equation}
\label{eq:trigger_collar_reg_app}
\begin{cases}
(\eta_1 - \beta_1) \left( E_1  \Phi(d_{\beta_1}(P_{CR}^{\ast},F)) \right. \\
\qquad \qquad \qquad \left. - G_1 \left( \Phi(d_{\beta_1}(P_{CR}^{\ast},F)) - \Phi(d_{\beta_1}(P_{CR}^{\ast},C)) \right) - J_1^{(\omega)} \right) {P_{CR}^{\ast}}^{\beta_1} \\
\qquad -(\eta_1 - \beta_2) \left( G_2 \left( \Phi(d_{\beta_2}(P_{CR}^{\ast},F)) - \Phi(d_{\beta_2}(P_{CR}^{\ast},C)) \right)  \right. \\
\qquad \qquad \qquad \left. + H_2 \Phi(d_{\beta_2}(P_{CR}^{\ast},C)) \right) {P_{CR}^{\ast}}^{\beta_2} \\
\qquad + (\eta_1 - 1)\dfrac{ P_{CR}^\ast Q}{r-\mu}e^{-(r-\mu)T} \left(1- \left( \Phi(d_1(P_{CR}^{\ast},F)) - \Phi(d_1(P_{CR}^{\ast},C)) \right) \right) \\
\qquad + \eta_1 \dfrac{F Q}{r}\left( 1-e^{-rT} \left(1- \Phi(d_0(P_{CR}^{\ast},F)) \right) \right) \\
\qquad - \eta_1 \left(\dfrac{C Q}{r} e^{-rT} \left(1- \Phi(d_0(P_{CR}^{\ast},C)) \right)+ I \right) = 0 & \textrm{ for } P_{CR}^\ast < F \\
\\
-(\eta_1 - \beta_1) \left( E_1  \left(1 - \Phi(d_{\beta_1}(P_{CR}^{\ast},F)) \right) \right. \\
\qquad \qquad \qquad \left. - G_1 \left(1-\left( \Phi(d_{\beta_1}(P_{CR}^{\ast},F)) - \Phi(d_{\beta_1}(P_{CR}^{\ast},C)) \right) \right) + J_1^{(\omega)} \right) {P_{CR}^{\ast}}^{\beta_1} \\
\qquad + (\eta_1 - \beta_2) \left( G_2 \left(1 - \left( \Phi(d_{\beta_2}(P_{CR}^{\ast},F)) - \Phi(d_{\beta_2}(P_{CR}^{\ast},C)) \right) \right)  \right. \\
\qquad \qquad \qquad \left. - H_2 \Phi(d_{\beta_2}(P_C^{\ast},C)) \right) {P_{CR}^{\ast}}^{\beta_2} \\
\qquad + (\eta_1 - 1)\dfrac{ P_{CR}^\ast Q}{r-\mu} \left(1+ e^{-(r-\mu)T} \left(1 - \left( \Phi(d_1(P_{CR}^{\ast},F)) - \Phi(d_1(P_{CR}^{\ast},C)) \right) \right) \right) \\
\qquad - \eta_1 \dfrac{F Q}{r}\left( 1-e^{-rT} \left(1- \Phi(d_0(P_{CR}^{\ast},F)) \right) \right) \\
\qquad - \eta_1 \left(\dfrac{C Q}{r} e^{-rT}  \Phi(d_0(P_{CR}^{\ast},C))+ I \right) = 0 & \textrm{ for } F \leqslant P_{CR}^\ast < C \\
\\
- (\eta_1 - \beta_1) \left( E_1  \Phi(d_{\beta_1}(P_{CR}^{\ast},F)) \right. \\
\qquad \qquad \qquad \left. + G_1 \left( \Phi(d_{\beta_1}(P_{CR}^{\ast},F)) - \Phi(d_{\beta_1}(P_{CR}^{\ast},C)) \right) + J_1^{(\omega)} \right) {P_{CR}^{\ast}}^{\beta_1} \\
\qquad + (\eta_1 - \beta_2) \left( G_2 \left( \Phi(d_{\beta_2}(P_{CR}^{\ast},F)) - \Phi(d_{\beta_2}(P_{CR}^{\ast},C)) \right)  \right. \\
\qquad \qquad \qquad \left. - H_2 \left( 1 - \Phi(d_{\beta_2}(P_{CR}^{\ast},C)) \right)\right) {P_{CR}^{\ast}}^{\beta_2} \\
\qquad + (\eta_1 - 1)\dfrac{ P_{CR}^\ast Q}{r-\mu}e^{-(r-\mu)T} \left(1- \left( \Phi(d_1(P_{CR}^{\ast},F)) - \Phi(d_1(P_{CR}^{\ast},C)) \right) \right) \\
\qquad - \eta_1 \dfrac{F Q}{r} e^{-rT} \left(1- \Phi(d_0(P_C^{\ast},F)) \right) \\
\qquad + \eta_1 \left(\dfrac{C Q}{r} \left(1- e^{-rT} \left(1- \Phi(d_0(P_{CR}^{\ast},C)) \right) \right)- I \right) = 0 & \textrm{ for } P_{CR}^\ast \geqslant C
\end{cases}
\end{equation}

\subsection{Minimum price guarantee} 
\label{app:floor}

For the case without regulatory uncertainty, Equation \eqref{eq:trigger_floor} results in the following two equations:
\begin{equation}
\label{eq:trigger_floor_app}
\begin{cases}
 - (\beta_1 - \beta_2) M_2 {P_M^{\ast}}^{\beta_2}\Phi(d_{\beta_2}(P_M^{\ast},F)) \\
\qquad +(\beta_1 - 1)\dfrac{ P_M^\ast Q}{r-\mu}e^{-(r-\mu)T} (1-\Phi(d_1(P_M^{\ast},F))) \\
\qquad + \beta_1 \left(\dfrac{F Q}{r}(1-e^{-rT}(1-\Phi(d_0(P_M^{\ast},F)))) - I \right) = 0 & \textrm{ for } P_M^\ast < F \\
\\
(\beta_1 - \beta_2) M_2 {P_M^{\ast}}^{\beta_2}(1-\Phi(d_{\beta_2(P_M^{\ast},F)})) \\ 
\qquad +(\beta_1 - 1)\dfrac{P_M^\ast Q}{r-\mu}\left(1+e^{-(r-\mu)T}(1-\Phi(d_1(P_M^{\ast},F))) \right)  \\
\qquad - \beta_1 \left(\dfrac{F Q}{r}e^{-rT}(1-\Phi(d_0(P_M^{\ast},F))) + I \right) = 0 & \textrm{ for } P_M^\ast \geqslant F
\end{cases}
\end{equation}

For the case with regulatory uncertainty, Equation \eqref{eq:trigger_floor_reg} results in the following two equations:
\begin{equation}
\label{eq:trigger_floor_reg_app}
\begin{cases}
(\eta_1 - \beta_1)(L_1 \Phi(d_{\beta_1})-N_1^{(\omega)}){P_{MR}^\ast}^{\beta_1}-(\eta_1-\beta_2)M_2{P_{MR}^\ast}^{\beta_2}\Phi(d_{\beta_2}) \\ 
\qquad +\eta_1\left( \dfrac{F Q}{r} -  \dfrac{F Q}{r}e^{-rT}(1-\Phi(d_{0}))-I\right) \\
\qquad  +(\eta_1-1) \dfrac{P_{MR}^{\ast} Q}{r-\mu}e^{-(r-\mu)T}(1-\Phi(d_{1})) = 0 & \textrm{ for } P_{MR}^\ast < F \\
\\
- (\eta_1-\beta_1)( L_1(1- \Phi(d_{\beta_1})) + N_1^{(\omega)}){P_{MR}^\ast}^{\beta_1} 
+ (\eta_1-\beta_2)M_2{P_{MR}^\ast}^{\beta_2}(1-\Phi(d_{\beta_2})) \\ 
\qquad  +(\eta_1-1) \dfrac{P_{MR}^{\ast} Q}{r-\mu}(1-e^{-(r-\mu)T}\Phi(d_{1}) \\
\qquad  + e^{-(r-\mu)T}) - \eta_1 \left(\dfrac{F Q}{r}e^{-rT}(1-\Phi(d_{0})) + I\right) = 0 & \textrm{ for } P_{MR}^\ast \geqslant F
\end{cases}
\end{equation}

\section{Proofs of relations between investment thresholds}
\label{app:proofs_relations}

\begin{proposition}
\label{prop:ThresholdCollarEqualPriceFloor}
The investment threshold of a project with perpetual sliding premium  is equal to the investment threshold with a perpetual minimum price guarantee when the price cap $C \to + \infty$.

\end{proposition}

\textbf{Proof:} 
We first calculate the limit of the value of the project with a perpetual sliding premium  (i.e., Equation \eqref{eq:ProjectValueC}) when the price cap $C \to + \infty$.

As \eqref{eq:ProjectValueC} depends on $E_1$, $G_1$, $G_2$ and $H_2$, we calculate the limit for each of these variables. In addition, it is easy to see that when $C \to + \infty$, the third region (i.e: when $P \geqslant C$) disappears. Hence, we do not calculate the limit of $H_2$ because it appears only in the third region \footnote{Recall that $G_2$ is equal to $M_2$.}.

\begin{align}
& \lim_{C\to+\infty} E_1 = \dfrac{ F^{1-\beta_1}Q} {\beta_1-\beta_2}\left(\frac{\beta_2}{r}-\frac{\beta_2 - 1}{r-\mu}\right) = L_1 \\
& \lim_{C\to+\infty} G_1 = 0 \\
& \lim_{C\to+\infty} G_2 = G_2 = M_2
\end{align}

Hence, the value of the project with a perpetual sliding premium  when the price cap $C \to + \infty$ is:

\begin{equation}
V_{CP}(P) =  
\begin{cases}
L_1 P^{\beta_1} + \dfrac{F Q }{r} & \textrm{ for }  P < F  \\
\ \\
M_2 P^{\beta_2} + \dfrac{P Q}{r - \mu} & \textrm{ for }  P \geqslant F  
\end{cases} 
= V_{MP}(P)
\end{equation}

We demonstrate that when $C \to + \infty$, the value of the project with a perpetual sliding premium  is equal to the value of the project with a perpetual minimum price guarantee. Therefore, it is easy to see that when $C \to + \infty$, the optimal threshold of the project with a perpetual sliding premium  is also equal to the optimal threshold of the project with a minimum price guarantee contract for the same sunk cost $I$.

\begin{proposition}
\label{prop:ThresholdCollarEqualFixedPrice}
The investment threshold of a project with perpetual sliding premium  is equal to the investment threshold with fixed price scheme when the price floor $F$ is equal to the price cap $C$.
\end{proposition}

\textbf{Proof:} 
We first calculate the value of the project with a perpetual sliding premium  (i.e., Equation \eqref{eq:ProjectValueC} when the price cap $C$ is equal to the price floor $F$.

As \eqref{eq:ProjectValueC} depends on $E_1$, $G_1$, $G_2$ and $H_2$, we first calculate the value of these variables when $C = F$. It is easy to see that $E_1$ and $H_2$ are equal to $0$, hence the first and third regions (i.e: when $P < F$ and $P \geqslant C$) are equal $\frac{F Q} {r} $. In addition, substituting $F$ for $P$ and $C$ in the second region also yields to $\frac{F Q} {r} $. Consequently, when $C = F$, $V_{CP}(P) = \frac{F Q} {r}$

Next, as $\Phi(-d_{\beta}(P,X)) = 1 - \Phi(d_{\beta}(P,X))$, substituting $F$ for $P$ and $C$ in \eqref{eq:SC} yields:

\begin{equation}
\label{eq:SC1}
S_C(P) =   \dfrac{F Q }{r}e^{-rT} 
\end{equation}

Therefore, substituting $\frac{F Q }{r}e^{-rT}$ for $S_C(P)$ and $\frac{F Q }{r}$ for $V_{CP}(P)$ in \eqref{eq:VALC} (i.e.: the value of the project with a finite collar scheme) yields:

\begin{equation}
\label{eq:VALC1}
V_{C} (P) = \dfrac{F Q }{r} - \dfrac{F Q }{r}e^{-rT} + \dfrac{P Q}{r-\mu}e^{-(r-\mu)T} = \dfrac{F Q }{r}(1 - e^{-rT}) + \dfrac{P Q}{r-\mu}e^{-(r-\mu)T}\end{equation}

In conclusion, the value of the project with a finite collar scheme in \eqref{eq:VALC} (i.e.: when $C=F$) is equal to \eqref{eq:VFfixed} (i.e.: the expected value of the project when there is a fixed price scheme). Thus, it is straightforward to see that the investment trigger of the project with a finite collar scheme is equal to the investment trigger of the project when there is a fixed price scheme.

\vspace{1.0cm}

\textbf{Proof of  Proposition \ref{prop:VMPandVCP1}:} In order to prove the result, we first prove that $V_{MP}(P) \geqslant V_{CP}(P)$ for any choice of the parameters, i.e., we prove first that the value of the project with a minimum price guarantee contract is always greater than the value of the project with sliding premium  contract, assuming both projects have the same value of the price floor.

We need to distinguish the following cases, depending on the relation between the price $P$, the minimum price $F$ and the price cap $C$, as the functions take different forms in each case.
\begin{itemize}
\item For $P<F$: in order to prove the result, we just need to prove that $L_1 \geqslant E_1 $. This holds as re-writing $E_1$ yields:
\begin{equation}
\label{eq:E1P}
E_1 = \dfrac{(F^{1-\beta_1}-C^{1-\beta_1})Q}{\beta_1-\beta_2} \left(\dfrac{\beta_2} {r} - \dfrac{\beta_2 - 1} {r - \mu}\right)=- \dfrac{C^{1-\beta_1}Q}{\beta_1-\beta_2} \left(\dfrac{\beta_2} {r} - \dfrac{\beta_2 - 1} {r - \mu}\right) + L_1 
\end{equation}

As $\mu < r$, $\beta_1 > 1$ and $\beta_2 < 0$, it is easy to see that $L_1$ is greater than $E_1$, because $\beta_2 \mu - r < 0$. Therefore, when $P < F$, the value of the project with a minimum price guarantee contract is greater than the value of the project with a sliding premium  contract.
\item For $F \leqslant P<C$: we prove that  $M_2 P^{\beta_2} \geqslant G_1 P^{\beta_1} + G_2 P^{\beta_2} $.

Re-writing $G_1$ yields:
\begin{equation}
\label{eq:L1P}
G_1 = -\dfrac{C^{1 - \beta_1} Q } {\beta_1 - \beta_2} \left(\dfrac{\beta_2} {r}-\dfrac{\beta_2 - 1} {r - \mu}\right) < 0
\end{equation}
as $\mu < r$, $\beta_1 > 1$ and $\beta_2 < 0$. Therefore the result follows, as $G_2=M_2$.

\item For $P \geqslant C$: it holds that 
$M_2 P^{\beta_2} + \dfrac{P Q}{r - \mu} \geqslant H_2 P^{\beta_2} + \dfrac{C Q}{r}$, as
re-writing  $H_2$ yields:
\begin{equation}
\label{eq:M2P}
H_2 = \dfrac{(F ^{1 -\beta_2} - C^{1-\beta_2})Q}{\beta_1-\beta_2} \left(\dfrac{\beta_1} {r} - \dfrac{\beta_1 - 1} {r - \mu}\right) = -\dfrac{C^{1-\beta_2}Q}{\beta_1-\beta_2} \left(\dfrac{\beta_1} {r}-\dfrac{\beta_1 - 1} {r - \mu}\right) + M_2
\end{equation}

It is easy to see that $\beta_1 \mu - r \leqslant 0$, because of the quadratic equation $\dfrac{1}{2} \sigma^2 \beta (\beta - 1) + \beta \mu - r = 0$. As $\beta_1 \geqslant 1$, the term $\beta_1 \mu - r$ is negative. Hence, $H_2 \leqslant M_2$. In addition, as $P \geqslant C$ because this is the third branch, it is straight forward to see that $\dfrac{P Q}{r - \mu} \geqslant \dfrac{C Q}{r}$. 
\end{itemize}

It is straight forward to see that the value of the project with a finite minimum price guarantee contract (i.e.: Equation \eqref{eq:VALM}) is also greater than the value of the project with a finite sliding premium  contract (i.e.: Equation \eqref{eq:VALC}). In other words, $V_{MP}(P) - S_M(P) + \dfrac{P Q}{r-\mu}e^{-(r-\mu)T} \geqslant V_{CP}(P) - S_C(P) + \dfrac{P Q}{r-\mu}e^{-(r-\mu)T}$, which implies, in its turn, that $S_{M} (P) \geqslant S_{C} (P)$ which, in each case, represents the value of the project with a delayed perpetual feed-in tariff contract.  However, the difference between $V_{M}$ and $V_{C}$ is greater than the difference between $S_{M}$ and $S_{C}$, because the former also has revenue before time $T$ and the latter only starts to have cash flow after $T$. In conclusion, the value of the project with a finite minimum price guarantee contract is always greater than the value of the project with a finite sliding premium  contract. 

\section{Proofs for the comparative statistics}
\label{app:proof_comparative_statics}

In this appendix we present the proof of Proposition \ref{prop:comparative_statics} regarding the comparative statics on the effect of $\lambda$ and $\omega$ grouped by FIT scheme. The proofs are presented briefly, as they follow from long calculations, starting with the simplest cases.

\subsection{Fixed-premium FIT}

The partial derivatives of the investment threshold is found deriving Equation \eqref{eq:trigger_premium}. Let us denote it by $E_P$.

The partial derivative of $P^\ast_{PR}$ to a parameter $z$, given that $E_P$ and $P^\ast_{PR}$ are both a function of $z$, is:
\begin{equation}
	\dfrac{\partial P^\ast_{PR}}{\partial z}=-\dfrac{\dfrac{\partial E_P}{\partial z}}{\dfrac{\partial E_P}{\partial P^\ast_{PR}}}
\end{equation}

\begin{equation}
	\dfrac{\partial E_P}{\partial P^\ast_{PR}} = \left(-\beta_1 ( \eta_1-\beta_1)V_1^{(\omega)}{P^\ast_{PR}}^{\beta_1} + (\eta_1-1)\dfrac{P_{PR}^{\ast} Q}{r - \mu} \right) \dfrac{1}{P^\ast_{PR}}
\end{equation}

From the smooth-pasting condition used to obtain Equation \eqref{eq:trigger_premium} we know that
\begin{equation}
	\beta_1 V_1^{(\omega)}{P^\ast_{PR}}^{\beta_1} = \dfrac{P_{PR}^{\ast}Q}{r - \mu}-\eta_1 W_1{P^\ast_{PR}}^{\eta_1},
\end{equation}
which allows us to obtain:
\begin{equation}
	\dfrac{\partial E_P}{\partial P^\ast_{PR}} =\left( (\eta_1 - \beta_1) \eta_1 W_1{P^\ast_{PR}}^{\eta_1} +(\beta_1-1) \dfrac{P_{PR}^{\ast} Q}{r - \mu} \right) \dfrac{1}{P^\ast_{PR}} >0.
\end{equation}

Since $W_1>0$, $\eta_1>\beta_1>1$, this partial derivative is always positive. Therefore, the sign of the partial derivative of the threshold to $z$ is the opposite of the sign of $\frac{\partial E_P}{\partial z}$. It can be shown that these partial derivatives are:
\begin{align}
\dfrac{\partial E_P}{\partial \lambda} &= -W_1 {P^\ast_{PR}}^{\beta_1} \dfrac{\partial \eta_1}{\partial \lambda} >0 \\
\dfrac{\partial E_P}{\partial \omega} &= -(\eta_1 - \beta_1)\dfrac{\partial V_1^{(\omega)}}{\partial \omega} {P^\ast_{PR}}^{\beta_1} \nonumber \\
&= -(\eta_1 - \beta_1) \dfrac{F Q}{r}\left(1-e^{-rT}\right) \left(\dfrac{P^\ast_{PR}}{{P^\ast_{P}}^{(\omega)}} \right)^{\beta_1}< 0 
\end{align}

The signs are obtained knowing that $\frac{\partial \eta_1}{\partial \lambda}<0$, $\eta_1>\beta_1>1$, $P^\ast_{PR} \leqslant {P^\ast_{P}}^{(\omega)}$.

\subsection{Fixed-price FIT}

Similarly to the previous section, the partial derivatives of the investment threshold is found deriving Equation \eqref{eq:trigger_fixed}. Denoting it by $E_F$ and showing that
\begin{equation}
	\dfrac{\partial E_P}{\partial P^\ast_{PR}} =\left( (\eta_1 - \beta_1) \eta_1 U_1{P^\ast_{PR}}^{\eta_1} +(\beta_1-1) \dfrac{Q}{r - \mu}P_{PR}^{\ast} \right) \dfrac{1}{P^\ast_{PR}} >0,
\end{equation}
the sign of the partial derivative of the threshold to $z$ is the opposite of the sign of $\frac{\partial E_P}{\partial z}$. It can be shown that these partial derivatives are:
\begin{align}
\dfrac{\partial E_F}{\partial \lambda} &= -U_1 {P^\ast_{FR}}^{\beta_1} \dfrac{\partial \eta_1}{\partial \lambda} >0 \\
\dfrac{\partial E_F}{\partial \omega} &= -(\eta_1 - \beta_1) \dfrac{\partial S_1^{(\omega)}}{\partial \omega} {P^\ast_{FR}}^{\beta_1} \nonumber \\
&= -(\eta_1 - \beta_1) \dfrac{F Q}{r}\left(1-e^{-rT}\right) \left(\dfrac{P^\ast_{FR}}{{P^\ast_{F}}^{(\omega)}} \right)^{\beta_1}< 0
\end{align}

The signs are obtained knowing additionally that $U_1>0$, $\frac{\partial \eta_1}{\partial \lambda}<0$, $\eta_1>\beta_1>1$, $P^\ast_{FR} \leqslant {P^\ast_{F}}^{(\omega)}$.

\subsection{Minimum price guarantee}

Similarly to the previous sections, the partial derivatives of the investment threshold is found deriving the two equations in \eqref{eq:trigger_floor_app}. Denoting them by $E_M$ and showing that
\begin{equation}
\dfrac{\partial E_M}{\partial P^\ast_{PR}} =\dfrac{1}{P_{MR}^\ast}
\begin{cases}
\eta_1 (\eta_1 - \beta_1) R_1 {P_{MR}^\ast}^{\eta_1} -\beta_2 (\beta_1-\beta_2)M_2{P_{MR}^\ast}^{\beta_2}N(d_{\beta_2}) \\
\qquad + (\beta_1-1) \dfrac{P_{MR}^{\ast} Q}{r-\mu}e^{-(r-\mu)T}(1-N(d_{1})) >0 & \textrm{ for } P_{MR}^\ast < F \\
\\
\eta_1 (\eta_1 - \beta_1) R_1 {P_{MR}^\ast}^{\eta_1} \\
\qquad -\beta_2 ((\eta_1-\beta_1)+(\eta_1-\beta_2))M_2{P_{MR}^\ast}^{\beta_2}(N(d_{\beta_2}) -1)\\
\qquad + (\beta_1-1) \dfrac{P_{MR}^{\ast} Q}{r-\mu} (1+e^{-(r-\mu)T}(1-N(d_{1}))) >0 & \textrm{ for } P_{MR}^\ast \geqslant F \\
\end{cases}
\end{equation}
the sign of the partial derivative of the threshold to $z$ is the opposite of the sign of $\frac{\partial E_M}{\partial z}$. It can be shown that these partial derivatives are:
\begin{align}
\dfrac{\partial E_M}{\partial \lambda} &= -R_1 {P^\ast_{MR}}^{\beta_1} \dfrac{\partial \eta_1}{\partial \lambda} >0 \\	
\dfrac{\partial E_M}{\partial \omega} &= -(\eta_1 - \beta_1)\dfrac{\partial N_1^{(\omega)}}{\partial \omega} {P^\ast_{MR}}^{\beta_1} \nonumber \\
&= -(\eta_1 - \beta_1) \dfrac{1}{\omega} \left( I-(\beta_1-1) N_1^{(\omega)}{{P^\ast_{M}}^{(\omega)}}^{\beta_1} \right) \left(\dfrac{{P^\ast_{MR}}}{{P^\ast_{M}}^{(\omega)}}\right)^{\beta_1}<0
\end{align}

The signs are obtained knowing that $\frac{\partial \eta_1}{\partial \lambda}<0$, $\eta_1>\beta_1>1$, $P^\ast_{MR} \leqslant {P^\ast_{M}}^{(\omega)} \leqslant P^\ast_{M} < P^\ast_{W}$, $ 0<R_1 \leqslant N_1^{(\omega)} \leqslant H_1 < {A_1}_W$, and
\begin{align*}
	&(\beta_1-1) N_1^{(\omega)}{{P^\ast_{M}}^{(\omega)}}^{\beta_1}<(\beta_1-1) H_1 {{P^\ast_{M}}}^{\beta_1} < (\beta_1-1) {A_1}_W {{P^\ast_{M}}}^{\beta_1} = (\beta_1-1) \dfrac{I}{\beta_1-1} \left(\dfrac{{P^\ast_{M}}}{P^\ast_{W}}\right)^{\beta_1} < I.
\end{align*}

\subsection{Sliding premium with cap and floor}

The partial derivatives of the investment threshold is found deriving the three Equations that result from \eqref{eq:trigger_collar} and that are presented in Equation \eqref{eq:trigger_collar_app} in Appendix \ref{app:collar}. Let us denote these equations as $E_C$.

As for the previous cases, it can be shown that $\frac{\partial E_C}{\partial P^\ast_{CR}} > 0$. Therefore, the sign of the partial derivative of the threshold to $z$ is the opposite of the sign of $\frac{\partial E_C}{\partial z}$. It can be shown that these partial derivatives are:
\begin{align}
\dfrac{\partial E_C}{\partial \lambda} &= - K_1 {P^\ast_{MR}}^{\beta_1}  \dfrac{\partial \eta_1}{\partial \lambda} >0 \\	
\dfrac{\partial E_C}{\partial \omega} &= -(\eta_1 - \beta_1)\dfrac{\partial J_1^{(\omega)}}{\partial \omega} {P^\ast_{MR}}^{\beta_1} \nonumber \\
&= -(\eta_1 - \beta_1) \dfrac{1}{\omega} \left( I-(\beta_1-1) J_1^{(\omega)}{{P^\ast_{C}}^{(\omega)}}^{\beta_1} + f\right) \left(\dfrac{{P^\ast_{CR}}}{{P^\ast_{C}}^{(\omega)}}\right)^{\beta_1}<0
\end{align}
where
\begin{equation}
	f = 
	\begin{cases}
		(\beta_1 - 1)  G_1^{(\omega)}  \Phi(d_{\beta_1}({P_C^{\ast}}^{(\omega)},F))  \\
		\qquad- (\beta_2 - 1) \left( H_2^{(\omega)}-G_2^{(\omega)} \right) \Phi(d_{\beta_2}({P_C^{\ast}}^{(\omega)},C)) \\
		\qquad + \dfrac{\omega_c C Q}{r} e^{-rT}  \Phi(d_0(P_C^{\ast},C)) & \textrm{ for } {P_C^{\ast}}^{(\omega)},F) < C \\
		\\
		-(\beta_1 - 1)  G_1^{(\omega)}  \left(1 - \Phi(d_{\beta_1}({P_C^{\ast}}^{(\omega)},F)) \right)  \\ 
		\qquad + (\beta_2 - 1) \left( H_2^{(\omega)}-G_2^{(\omega)} \right) \left(1-\Phi(d_{\beta_2}({P_C^{\ast}}^{(\omega)},C)) \right) \\
		\qquad - \dfrac{\omega_c C Q}{r} \left(1-e^{-rT}  \Phi(d_0(P_C^{\ast},C)) \right) & \textrm{ for } {P_C^{\ast}}^{(\omega)},F) \geqslant C
		
	\end{cases}
\end{equation}

The signs are obtained knowing that $\frac{\partial \eta_1}{\partial \lambda}<0$, $\eta_1>\beta_1>1$, $P^\ast_{CR} \leqslant {P^\ast_{C}}^{(\omega)}$, and
\begin{align*}
	 I-(\beta_1-1) J_1^{(\omega)}{{P^\ast_{C}}^{(\omega)}}^{\beta_1} + f > 0.
\end{align*}

The sign of this equation in proved by deriving the value of option to invest in a sliding premium FIT without the floor (a cap FIT), denoted by $Z_1^{(\omega)} P^\beta_1$ (for instance taking $C$ to the limit of $0$). It can be shown that $f> (\beta_1-1) Z_1^{(\omega)}{{P^\ast_{C}}^{(\omega)}}^{\beta_1} - I$ and $Z_1^{(\omega)} > J_1^{(\omega)}$.

\end{appendix}

\newpage

\bibliographystyle{jmr}
\bibliography{references}

\end{document}